\documentclass[onefignum,onetabnum]{siamart190516}

\usepackage{enumitem}
\usepackage{amssymb}
\usepackage{xargs}
\usepackage{tikz}
\usepackage{lipsum}
\usepackage{amsfonts}
\usepackage{graphicx}
\usepackage{epstopdf}
\usepackage{algorithmic}
\usepackage{booktabs}
\usepackage{tabularx}
\ifpdf
  \DeclareGraphicsExtensions{.eps,.pdf,.png,.jpg}
\else
  \DeclareGraphicsExtensions{.eps}
\fi

\usepackage{array}

\usepackage[sort]{cite}

\newsiamthm{assumption}{Assumption}
\newsiamremark{remark}{Remark}
\newsiamthm{main result}{Main Result}

\headers{Approximation by Deep Rectifier Networks}{M.\ Ali,\ A.\ Nouy}

\title{Approximation of Smoothness Classes
by Deep \alert{Rectifier} Networks\thanks{Submitted to the editors 17th of August, 2020.
\funding{The authors acknowledge AIRBUS Group for the financial support with the project AtRandom.}}}

\author{Mazen Ali\thanks{Centrale Nantes, LMJL UMR CNRS 6629, France
(\email{\{mazen.ali, anthony.nouy\}@ec-nantes.fr}).}
\and Anthony Nouy\footnotemark[2]}

\usepackage{amsopn}

\newcommand{\alert}[1]{#1}
\newcommand{\alertt}[1]{#1}
\newcommand{\bs}{\boldsymbol}
\newcommand{\mc}{\mathcal}
\newcommand{\mb}{\mathbb}
\renewcommand*\d{\mathop{}\!\mathrm{d}}

\newcommand{\R}{\mathbb{R}}
\newcommand{\N}{\mathbb{N}}
\newcommand{\Z}{\mathbb{Z}}
\newcommand{\id}{\mathbb{I}}

\newcommand{\repu}{\ensuremath{\mathsf{RePU}}}
\newcommand{\relu}{\ensuremath{\mathsf{ReLU}}}
\newcommand{\rlz}{\ensuremath{\mathcal{R}}}
\newcommandx{\norm}[2][2=]{\ensuremath{\left\| #1 \right\|_{#2}}}
\newcommandx{\snorm}[2][2=]{\ensuremath{\left| #1 \right|_{#2}}}
\DeclareMathOperator*{\supp}{supp}
\DeclareMathOperator*{\depth}{depth}
\DeclareMathOperator*{\dist}{dist}
\newcommand{\aff}{\ensuremath{\mathsf{Aff}}}
\newcommand{\nl}{\ensuremath{\mathsf{NL}}}
\newcommand{\tool}{\ensuremath{\Sigma_n}}

\ifpdf
\hypersetup{
  pdftitle={Approximation of Smoothness Classes by Deep Rectifier Networks},
  pdfauthor={M. Ali and A. Nouy}
}
\fi

\begin{document}

\maketitle

\begin{abstract}
    We consider approximation rates
    of sparsely connected deep rectified linear
    unit (ReLU) and rectified power unit (RePU) neural networks
    for functions in Besov spaces $B^\alpha_{q}(L^p)$
    in arbitrary dimension $d$,
    on \alert{general} domains.
    We show that \alert{deep rectifier} networks with a fixed
    activation function attain \alert{optimal or near to optimal}
    approximation rates for functions
    in the Besov space $B^\alpha_{\tau}(L^\tau)$ on the critical embedding
    line $1/\tau=\alpha/d+1/p$ for \emph{arbitrary} smoothness order
    $\alpha>0$.
    Using interpolation theory,
    this implies that the entire range of smoothness classes
    at or above the critical line is (near to) optimally approximated
    by deep ReLU/RePU networks.
\end{abstract}



\section{Introduction}\label{sec:intro}
Artificial neural networks (NNs) have become a popular
tool in various fields of computational and data science.
Due to \alert{their} popularity and good performance,
NNs motivated a lot of research in mathematics --
especially in recent years -- in an attempt to
explain the properties of NNs
responsible for their success.

Although many aspects of NNs still lack a satisfactory
mathematical explanation,
the \emph{expressivity} or approximation theoretic properties
of NNs are by now quite well understood.
By expressivity we mean the theoretical capacity of NNs
to approximate functions from different classes.
We do not intend to give a literature overview
on this topic and instead refer to the recent survey in
\cite{review}.

\subsection*{Contribution}
In this work, we contribute to the existing body of knowledge
on the expressivity of NNs by showing that
the very popular \alert{rectifier
NNs} can approximate a \alert{wide range
of smoothness classes in the Besov scale with (near to)} optimal complexity.
\alert{In the context of this work,
``optimality'' refers to the notion of
\emph{continuous nonlinear widths}\footnote{\alert{Closely related to
Aleksandrov widths, see \cite{Dung1996}.}} introduced in
\cite{devore1989optimal}. For the approximation of functions in the Besov space
$B^\alpha_q(L^p(\R^d))$,
an approximation tool with a continuous parameter selection for the approximand can achieve worst case approximation rates of at most $\alpha/d$.}
To make the distinction to existing results
clear, we briefly review what is
known by now about the approximation of
some more standard smoothness classes
closely related to our work.
In all instances ``complexity'' is measured by the number of
connections, i.e., non-zero weights.

In \cite{OSZ19_839}, it was shown that analytic functions
on a compact product domain in any dimension
can be approximated in the Sobolev norm
$W^{k,\infty}$
by ReLU and RePU networks with close to exponential convergence.
In \cite{PETERSEN2018296}, it was shown that ReLU networks can approximate
any Hölder continuous function with optimal
complexity.
In \cite{gribonval:hal-02117139}, it was shown that functions in the Besov space
$B^\alpha_p(L^p(\Omega))$ on bounded Lipschitz
domains $\Omega\subset\R^d$ in any dimension
can be approximated in the $L^p$-norm
by RePU networks with activation function of degree $r\gtrsim\alpha$
with optimal complexity.
The spaces $B^\alpha_p(L^p(\Omega))$ correspond to the vertical
line in \Cref{fig:DeVore}, i.e., for $p\geq 1$ these are either
the same or slightly larger than
the Sobolev spaces $W^{k,p}(\Omega)$.

\begin{figure}[h]
        \centering
        \includegraphics[scale=.5]{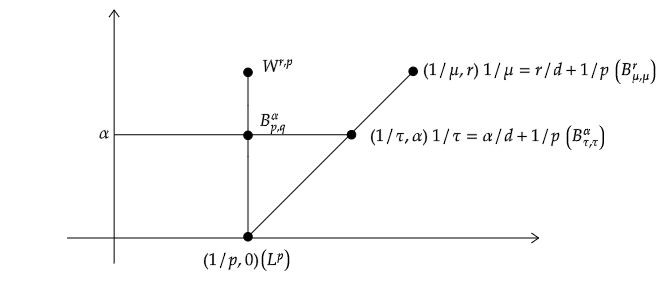}
        \caption{DeVore diagram of smoothness spaces \cite{devore_1998}.
        The Sobolev embedding line is the diagonal
        with the points $(1/\tau,\alpha)$ and $(1/\mu, r)$,
        \alert{i.e., for a fixed $p$
        and variable $1/\tau$,
        the diagonal is the line $\alpha
        =d(1/\tau-1/p)$
        with slope $d$ and offset $-d/p$.
        All points above the diagonal line correspond to Besov spaces
        compactly embedded in $L^p$,
        points on the line may or may not be continuously embedded
        in $L^p$, and points below the line are never embedded in $L^p$.}}
        \label{fig:DeVore}
\end{figure}

In \cite{SchwabFEM}, it was shown that functions in the
Besov space $B^\alpha_\tau(L^\tau(I))$, for
$\alpha>1/\tau-1/p$ on bounded intervals $I\subset\R$,
can be approximated in the \alert{(fractional) Sobolev $W^{s,p}(I)$-norm}
\alert{with deep ReLU networks}
with near to optimal complexity.
\alert{In \cite{suzuki2018adaptivity}, the author shows that
functions in $B^\alpha_\tau(L^\tau([0,1]^d))$
for $\alpha/d>1/\tau-1/p$ and in Besov spaces of dominating
mixed smoothness can be approximated in $L^p$ with
deep ReLU networks with near to
optimal complexity.}
The space \alert{$B^\alpha_\tau(L^\tau(\Omega))$
for $\alpha/d>1/\tau-1/p$
and Lipschitz domains $\Omega$}
is above the \emph{critical embedding line} of functions
that barely have enough regularity to be members of $L^p$,
see the diagonal in \Cref{fig:DeVore}. Spaces above
this critical line are embedded in $L^p$, spaces on
this line may or may not be embedded in $L^p$,
and spaces below this line are never embedded in $L^p$.

It was also shown in \cite{SchwabFEM}
that \alertt{piecewise} Gevrey functions
can be approximated with close to exponential convergence. Similar results for classical smoothness spaces of univariate functions are contained  in \cite{Daubechies2019arXiv}.

In this work, we show that functions in
\alert{isotropic Besov spaces}
$B^\alpha_{\alert{q}}(L^\tau(\Omega))$, for
\alert{$\alpha/d\geq 1/\tau-1/p$, $q\leq\tau$ and $\Omega\subset\R^d$ an
$(\varepsilon,\delta)$- or a Lipschitz domain (see
\Cref{def:Lipschitz,def:epsdelta})}
in any dimension $d\in\N$,
can be
approximated by RePU networks
with activation function of degree \alert{$r\geq2$}
with optimal complexity for any $\alpha>0$.
We show the same for ReLU networks with near
to\footnote{\alert{For} any approximation rate arbitrarily close
to optimal.} optimal complexity.
This completes the picture for \alert{rectifier networks}
expressivity rates for classical \alert{isotropic} smoothness spaces
in the sense that, with regard to $L^p$ approximation,
functions from any Besov space
\alertt{on or above the embedding line (see \Cref{fig:DeVore}),
with $q\leq\tau$,}
can be approximated by ReLU/RePU networks with
(near to) optimal complexity, \alert{universal in the smoothness order
$\alpha$.}

\subsection*{Outline}
We begin in
\Cref{sec:nn} and \Cref{sec:app}
by reviewing the theoretical framework
of our work.
We then state the main result in
\Cref{sec:main} \alert{that includes a summary of the results
on isotropic Besov spaces}.
To keep the presentation self-contained,
we review previous results
\alert{-- that we require for our work --} on ReLU approximation in
\Cref{sec:prelim} and \alert{Besov} smoothness classes in \Cref{sec:wavs}.
Finally, in \Cref{sec:optimal} we derive the main result of this work, stated again in
\Cref{thm:embedd}.
The reader familiar with results on ReLU/RePU approximation
and wavelet characterizations of Besov spaces can skip directly
to \Cref{sec:optimal}.

\subsection*{\alert{Notation and Terminology}} For quantities
$A,B\in\R$, we will use the notation
$A\lesssim B$ if there exists a constant $C$ that does not depend
on $A$ or $B$ such that $A\leq CB$. Similarly for $\gtrsim$ and $\sim$ if
both inequalities hold. \alert{We use $\N$ for natural numbers and
$\N_0:=\N\cup\{0\}$.}
We use $\supp(f)$ to denote the support of a function
$f\in\R^d\rightarrow\R$
\begin{align*}
    \supp(f):=\overline{\{x\in\R^d:\;f(x)\neq 0\}},
\end{align*}
and $|\supp(f)|$ to denote the Lebesgue measure of this set.
We use $\#$ to denote the standard counting measure.

\alert{We denote by $L^p(\Omega)$ the Lebesgue space of
real-valued $p$-integrable functions for $0<p\leq\infty$ on
open subsets $\Omega\subset\R^d$ and by $H_p(\Omega)$
the real Hardy space (see \cite{Fefferman1972}).
\alertt{Recall that the real Hardy spaces are isomorphic 
to $L^p$ for $p>1$.}
In this work, we will be referring to one of the following three types
of domains $\Omega$ (see \cite{Jones1981,RobertAdams2003,Stein70,devore1993besov}).}
\alert{
\begin{definition}[Special Lipschitz]\label{def:speclip}
    We call an open set $\Omega\subset\R^d$ a \emph{special Lipschitz domain} if
    there exists a Lipschitz function $\phi:\R^{d-1}\rightarrow\R$ with
    $|\phi(x_1)-\phi(x_2)|\leq M\|x_1-x_2\|_2$ for some constant $M>0$ such that
    \begin{align*}
        \Omega=\left\{(x,y):\;
        x:\R^{d-1},\;y\in\R\text{ and }y>\phi(x)\right\}.
    \end{align*}
\end{definition}
}
\alert{
\begin{definition}[Strong Local Lipschitz Condition]\label{def:Lipschitz}
    An open set $\Omega\subset\R^d$
    is said to satisfy the \emph{strong local Lipschitz condition}
    -- also known as a \emph{minimally smooth domain} or simply \emph{Lipschitz} --
    if there exists $\varepsilon>0$,
    \alertt{$M>0$, a locally finite open cover $\{U_i:\;i\in\N\}$
    of $\partial\Omega$,
    and, for each $i$ a real-valued function
    $f_i$ of $d-1$ variables, such that
    \begin{enumerate}[label=(\roman*)]
    	\item for some finite $R$, every collection of $R+1$
    	of the sets $U_i$ has empty intersection;
    	\item for every pair of points $x,y\in\Omega$ such that
    	$\dist(x,\partial\Omega),
    	\dist(y,\partial\Omega)>\varepsilon$ and
    	$\|x-y\|_2<\varepsilon$, there exists
    	$i$ such that
    	\begin{align*}
    		x,y\in V_i:=\left\{
    		z\in U_i:\dist(z,\partial U_i)>\varepsilon\right\};
    	\end{align*}    	
    	\item each function $f_i$ satisfies a Lipschitz condition
    	with constant $M$;    	
    	\item for some Cartesian coordinate system
    	$(\xi_{i,1},\ldots,\xi_{i,d})$ in $U_i$,
    	$\Omega\cap U_i$ is represented by the inequality
    	\begin{align*}
    		\xi_{i,d}<f_i(\xi_{i,1},\ldots,\xi_{i,d-1}).
    	\end{align*}
    \end{enumerate}
    }     
\end{definition}
}
\alert{
\begin{definition}[$(\varepsilon, \delta)$-Domain]\label{def:epsdelta}
    An open set $\Omega$ is called an $(\varepsilon, \delta)$-domain if
    for any $x, y\in\Omega$, satisfying $\alertt{\|x-y\|_2}\leq\delta$, there exists
    a rectifiable path $\Gamma\alertt{\subset\Omega}$ of length $\leq C_0\alertt{\|x-y\|_2}$
    \alertt{for some constant $C_0>0$},
    connecting $x$ and $y$, such that for each
    $z\in\Gamma$,
    $$\dist(z,\partial\Omega)\geq\varepsilon\min\left\{\alertt{\|z-x\|_2, \|z-y\|_2}\right\}.$$
\end{definition}}

\alert{
The inclusions between the different domain types are as follows
\begin{align*}
    \text{special Lipschitz} \Rightarrow
    \text{strong locally Lipschitz} \Rightarrow
    (\varepsilon, \delta)\text{-domain}.
\end{align*}
These domain types are not necessarily bounded and the special case
$\Omega=\R^d$ trivially satisfies the strong local Lipschitz condition.
}

\subsection{Neural Networks}\label{sec:nn}
We briefly introduce the mathematical description and notation we
use for NNs \alertt{throughout} this work. Specifically, we will only consider
feed-forward NNs. In \Cref{fig:ann}, we sketch a pictorial representation of
a feed-forward NN.

\begin{figure}[ht]
    \center
    \tikzset{every picture/.style={line width=0.75pt}} 

\begin{tikzpicture}[x=0.75pt,y=0.75pt,yscale=-.7,xscale=.7]

\draw    (355.33,27.33) -- (503.73,66.49) ;
\draw [shift={(505.67,67)}, rotate = 194.78] [color={rgb, 255:red, 0; green, 0; blue, 0 }  ][line width=0.75]    (10.93,-3.29) .. controls (6.95,-1.4) and (3.31,-0.3) .. (0,0) .. controls (3.31,0.3) and (6.95,1.4) .. (10.93,3.29)   ;
\draw    (355.33,124.33) -- (503.73,84.85) ;
\draw [shift={(505.67,84.33)}, rotate = 525.1] [color={rgb, 255:red, 0; green, 0; blue, 0 }  ][line width=0.75]    (10.93,-3.29) .. controls (6.95,-1.4) and (3.31,-0.3) .. (0,0) .. controls (3.31,0.3) and (6.95,1.4) .. (10.93,3.29)   ;
\draw    (355.33,233.33) -- (503.75,187.92) ;
\draw [shift={(505.67,187.33)}, rotate = 522.99] [color={rgb, 255:red, 0; green, 0; blue, 0 }  ][line width=0.75]    (10.93,-3.29) .. controls (6.95,-1.4) and (3.31,-0.3) .. (0,0) .. controls (3.31,0.3) and (6.95,1.4) .. (10.93,3.29)   ;
\draw    (355.33,233.33) -- (512.87,86.86) ;
\draw [shift={(514.33,85.5)}, rotate = 497.08] [color={rgb, 255:red, 0; green, 0; blue, 0 }  ][line width=0.75]    (10.93,-3.29) .. controls (6.95,-1.4) and (3.31,-0.3) .. (0,0) .. controls (3.31,0.3) and (6.95,1.4) .. (10.93,3.29)   ;
\draw    (355.33,124.33) -- (503.75,169.42) ;
\draw [shift={(505.67,170)}, rotate = 196.9] [color={rgb, 255:red, 0; green, 0; blue, 0 }  ][line width=0.75]    (10.93,-3.29) .. controls (6.95,-1.4) and (3.31,-0.3) .. (0,0) .. controls (3.31,0.3) and (6.95,1.4) .. (10.93,3.29)   ;
\draw    (231.33,84.33) -- (343.85,33.33) ;
\draw [shift={(345.67,32.5)}, rotate = 515.61] [color={rgb, 255:red, 0; green, 0; blue, 0 }  ][line width=0.75]    (10.93,-3.29) .. controls (6.95,-1.4) and (3.31,-0.3) .. (0,0) .. controls (3.31,0.3) and (6.95,1.4) .. (10.93,3.29)   ;
\draw    (231.33,84.33) -- (342.12,123.66) ;
\draw [shift={(344,124.33)}, rotate = 199.55] [color={rgb, 255:red, 0; green, 0; blue, 0 }  ][line width=0.75]    (10.93,-3.29) .. controls (6.95,-1.4) and (3.31,-0.3) .. (0,0) .. controls (3.31,0.3) and (6.95,1.4) .. (10.93,3.29)   ;
\draw    (231.33,167.33) -- (342.27,232.32) ;
\draw [shift={(344,233.33)}, rotate = 210.36] [color={rgb, 255:red, 0; green, 0; blue, 0 }  ][line width=0.75]    (10.93,-3.29) .. controls (6.95,-1.4) and (3.31,-0.3) .. (0,0) .. controls (3.31,0.3) and (6.95,1.4) .. (10.93,3.29)   ;
\draw    (231.33,84.33) -- (347.38,221.97) ;
\draw [shift={(348.67,223.5)}, rotate = 229.87] [color={rgb, 255:red, 0; green, 0; blue, 0 }  ][line width=0.75]    (10.93,-3.29) .. controls (6.95,-1.4) and (3.31,-0.3) .. (0,0) .. controls (3.31,0.3) and (6.95,1.4) .. (10.93,3.29)   ;
\draw    (231.33,167.33) -- (342.76,131.12) ;
\draw [shift={(344.67,130.5)}, rotate = 522] [color={rgb, 255:red, 0; green, 0; blue, 0 }  ][line width=0.75]    (10.93,-3.29) .. controls (6.95,-1.4) and (3.31,-0.3) .. (0,0) .. controls (3.31,0.3) and (6.95,1.4) .. (10.93,3.29)   ;
\draw    (231.33,167.33) -- (347.33,38.98) ;
\draw [shift={(348.67,37.5)}, rotate = 492.1] [color={rgb, 255:red, 0; green, 0; blue, 0 }  ][line width=0.75]    (10.93,-3.29) .. controls (6.95,-1.4) and (3.31,-0.3) .. (0,0) .. controls (3.31,0.3) and (6.95,1.4) .. (10.93,3.29)   ;
\draw    (80.34,209.67) -- (224.09,96.74) ;
\draw [shift={(225.67,95.5)}, rotate = 501.85] [color={rgb, 255:red, 0; green, 0; blue, 0 }  ][line width=0.75]    (10.93,-3.29) .. controls (6.95,-1.4) and (3.31,-0.3) .. (0,0) .. controls (3.31,0.3) and (6.95,1.4) .. (10.93,3.29)   ;
\draw    (80.34,126.67) -- (217.73,90.02) ;
\draw [shift={(219.67,89.5)}, rotate = 525.06] [color={rgb, 255:red, 0; green, 0; blue, 0 }  ][line width=0.75]    (10.93,-3.29) .. controls (6.95,-1.4) and (3.31,-0.3) .. (0,0) .. controls (3.31,0.3) and (6.95,1.4) .. (10.93,3.29)   ;
\draw    (80.34,47.67) -- (218.07,83.83) ;
\draw [shift={(220,84.33)}, rotate = 194.71] [color={rgb, 255:red, 0; green, 0; blue, 0 }  ][line width=0.75]    (10.93,-3.29) .. controls (6.95,-1.4) and (3.31,-0.3) .. (0,0) .. controls (3.31,0.3) and (6.95,1.4) .. (10.93,3.29)   ;
\draw    (80.34,211.67) -- (219.74,174.02) ;
\draw [shift={(221.67,173.5)}, rotate = 524.89] [color={rgb, 255:red, 0; green, 0; blue, 0 }  ][line width=0.75]    (10.93,-3.29) .. controls (6.95,-1.4) and (3.31,-0.3) .. (0,0) .. controls (3.31,0.3) and (6.95,1.4) .. (10.93,3.29)   ;
\draw    (80.34,126.67) -- (218.08,166.77) ;
\draw [shift={(220,167.33)}, rotate = 196.23] [color={rgb, 255:red, 0; green, 0; blue, 0 }  ][line width=0.75]    (10.93,-3.29) .. controls (6.95,-1.4) and (3.31,-0.3) .. (0,0) .. controls (3.31,0.3) and (6.95,1.4) .. (10.93,3.29)   ;
\draw  [draw opacity=0][fill={rgb, 255:red, 208; green, 2; blue, 27 }  ,fill opacity=1 ] (80.34,40) -- (88.68,47.67) -- (80.34,55.33) -- (72,47.67) -- cycle ;
\draw  [draw opacity=0][fill={rgb, 255:red, 208; green, 2; blue, 27 }  ,fill opacity=1 ] (80.34,119) -- (88.68,126.67) -- (80.34,134.33) -- (72,126.67) -- cycle ;
\draw  [draw opacity=0][fill={rgb, 255:red, 208; green, 2; blue, 27 }  ,fill opacity=1 ] (80.34,204) -- (88.68,211.67) -- (80.34,219.33) -- (72,211.67) -- cycle ;
\draw  [draw opacity=0][fill={rgb, 255:red, 248; green, 231; blue, 28 }  ,fill opacity=1 ] (220,84.33) .. controls (220,78.07) and (225.07,73) .. (231.33,73) .. controls (237.59,73) and (242.67,78.07) .. (242.67,84.33) .. controls (242.67,90.59) and (237.59,95.67) .. (231.33,95.67) .. controls (225.07,95.67) and (220,90.59) .. (220,84.33) -- cycle ;
\draw  [draw opacity=0][fill={rgb, 255:red, 248; green, 231; blue, 28 }  ,fill opacity=1 ] (220,167.33) .. controls (220,161.07) and (225.07,156) .. (231.33,156) .. controls (237.59,156) and (242.67,161.07) .. (242.67,167.33) .. controls (242.67,173.59) and (237.59,178.67) .. (231.33,178.67) .. controls (225.07,178.67) and (220,173.59) .. (220,167.33) -- cycle ;
\draw  [draw opacity=0][fill={rgb, 255:red, 248; green, 231; blue, 28 }  ,fill opacity=1 ] (344,27.33) .. controls (344,21.07) and (349.07,16) .. (355.33,16) .. controls (361.59,16) and (366.67,21.07) .. (366.67,27.33) .. controls (366.67,33.59) and (361.59,38.67) .. (355.33,38.67) .. controls (349.07,38.67) and (344,33.59) .. (344,27.33) -- cycle ;
\draw  [draw opacity=0][fill={rgb, 255:red, 248; green, 231; blue, 28 }  ,fill opacity=1 ] (344,124.33) .. controls (344,118.07) and (349.07,113) .. (355.33,113) .. controls (361.59,113) and (366.67,118.07) .. (366.67,124.33) .. controls (366.67,130.59) and (361.59,135.67) .. (355.33,135.67) .. controls (349.07,135.67) and (344,130.59) .. (344,124.33) -- cycle ;
\draw  [draw opacity=0][fill={rgb, 255:red, 248; green, 231; blue, 28 }  ,fill opacity=1 ] (344,233.33) .. controls (344,227.07) and (349.07,222) .. (355.33,222) .. controls (361.59,222) and (366.67,227.07) .. (366.67,233.33) .. controls (366.67,239.59) and (361.59,244.67) .. (355.33,244.67) .. controls (349.07,244.67) and (344,239.59) .. (344,233.33) -- cycle ;
\draw  [draw opacity=0][fill={rgb, 255:red, 126; green, 211; blue, 33 }  ,fill opacity=1 ] (505.67,67) -- (523,67) -- (523,84.33) -- (505.67,84.33) -- cycle ;
\draw  [draw opacity=0][fill={rgb, 255:red, 126; green, 211; blue, 33 }  ,fill opacity=1 ] (505.67,170) -- (523,170) -- (523,187.33) -- (505.67,187.33) -- cycle ;

\end{tikzpicture}
    \caption{Example of a feed-forward neural network. On the left we have the \emph{input} nodes marked in red that
represent input data to the network.
The yellow nodes are the \emph{neurons} that
perform some simple operations on the input.
The edges between the nodes represent \emph{connections} that transfer (after
possibly applying an affine transformation) the
output of one node into the input of another.
The final green nodes are the \emph{output} nodes.
In this particular example the
number of \emph{layers} $L$ is three, with two \emph{hidden layers}.}
    \label{fig:ann}
\end{figure}
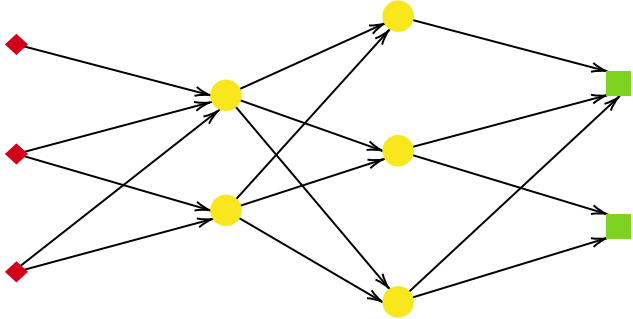

Input values are passed on to the first layer
of \emph{neurons} after possibly undergoing an affine transformation.
In the neurons, an \emph{activation} function
is applied to the transformed input values. The result again undergoes an affine
transformation and is passed
to the next layer and so on, until the output layer is reached.

The number of inputs and outputs
is typically determined by the intended application.
Specifying the architecture of such an NN amounts to choosing
the number of layers, the number of neurons in each hidden layer,
the activation functions and
the connections or, equivalently, the position
of the non-zero weights in the affine transformations.
The process of training then consists of determining said weights.

We formalize our description of the considered mathematical objects.
Let $L\in\N$ be the number of layers,
$N_0$ the number of inputs,
$N_L$ the number of outputs
and $N_1,\ldots, N_{L-1}$ the number of neurons
in each hidden layer.
A neural network $\Phi$ can be described by the tuple
\begin{align*}
    \Phi:=((T_1,\bs\sigma_1),\ldots,\alert{,(T_{L-1},\bs\sigma_{L-1}),(T_L)}),
\end{align*}
where for each $1\leq l\leq L$, $T_l$ is an affine transformation
\begin{align}\label{eq:affmaps}
    T_l:\R^{N_{l-1}}\rightarrow\R^{N_l},\quad
    x\mapsto A_l x+b_l,\quad
    A_l\in\R^{N_{l}\times N_{l-1}},\;b_l\in\R^{N_l},
\end{align}
and $\bs\sigma_l:\R^{N_l}\rightarrow\R^{N_l}$ is a (nonlinear)
function, usually applied component-wise
as $$x\mapsto(\sigma^{\alert{(1)}}_l(x_1),\ldots,\sigma^{\alert{(N_l)}}_l(x_{N_l})).$$
In this work we will use RePU activation functions, i.e.,
\begin{align}\label{eq:nl}
    \sigma^{\alert{(i)}}_l\in\{\id_\R,\rho_r\},\quad
    \rho_r(t):=\max\{0,t\}^{\alert{r}},\quad1\leq l\leq L-1,\;\alert{r\in\N.}
\end{align}
where $\id_\R:\R\rightarrow\R$ is the identity map and
for $r=1$, $\rho_1$ is referred to as the
\emph{rectified linear unit (ReLU)}.
We allow for the possibility of a non-strict network,
i.e., an activation function is either $\id_\R$ or $\rho_r$.
Another possibility is a \emph{strict} network where each activation
function is necessarily $\rho_r$ (with the exclusion of
the output nodes). But, as was shown in \cite{gribonval:hal-02117139},
the approximation theoretic properties of both are the same and thus,
for our work, it is irrelevant.

Let $\aff(N_{l-1},N_l)$ denote the set of affine maps as in
\eqref{eq:affmaps} and $\nl(N_l,r)$ denote the set of activation functions
as in \eqref{eq:nl}. For fixed $N_0$, $N_L$,
define
\begin{align*}
    &\repu^{r,N_0,N_L}:=\\
    &\bigcup_{L\in\N}\;\bigcup_{(N_1,\ldots,N_{\alert{L-1}})\in\N^{\alert{L-1}}}
     \aff(N_0,N_1)\times\nl(N_1,r)\times\cdots\times
      \alert{\nl(N_{L-1},r)\times\aff(N_{L-1},N_L)},\\
      &\relu^{N_0,N_L}:=\repu^{1,N_0,N_L},
\end{align*}
and the realization map $\rlz:\repu^{r,N_0,N_L}\rightarrow
(\R^{N_L})^{\R^{N_0}}$ by
\begin{align*}
    \rlz(\Phi):=\alert{T_L\circ\bs \sigma_{L-1}}\circ\cdots\bs \sigma_1\circ T_1.
\end{align*}

\subsection{Approximation Classes}\label{sec:app}
In this work, we will \alert{derive} results in the
approximation theoretic framework introduced
in \cite{gribonval:hal-02117139}.
Before we do so, let us first recall the definition
of approximation \alert{classes}.

Let $X$ be a quasi-normed linear space, $\tool\subset X$ subsets of $X$
for $n\in\N_0$ and $\Sigma:= (\tool)_{n\in \N_0}$ an approximation tool.
Define the best approximation error
\begin{align*}
E(f,{\tool})_{X}:=\inf_{\varphi\in\tool}\norm{f-\varphi}[X].
\end{align*}
With this we define approximation classes as

\begin{definition}[Approximation Classes]\label{def:asq}
    For any $f\in X$ and $\alpha>0$, define the quantity
    \begin{align*}
        \norm{f}[A^\alpha_q]:=
        \begin{cases}
            \left(\sum_{n=1}^\infty[n^\alpha E(f,{\tool})_X]^q\frac{1}{n}\right)^{1/q},&\quad 0<q<\infty,\\
            \sup_{n\geq 1}[n^\alpha E(f,{\tool})_X],&\quad q=\infty.
        \end{cases}
    \end{align*}
    The approximation classes $A^\alpha_q$
    of $\Sigma = (\tool)_{n\in \N_0}$ are defined by  
    \begin{align*}
       A^\alpha_q(X,\Sigma):=\left\{f\in X:\; \norm{f}[A^\alpha_q]<\infty\right\}.
    \end{align*}
\end{definition}
The utility of using these classes comes to light only if the sets
$\tool$ satisfy certain properties. This was discussed in detail
in \cite{gribonval:hal-02117139} and the relevant properties were shown to hold for
RePU networks.

We perform approximation in $X=L^p(\Omega)$ for $1<p\leq\infty$,
\alert{we comment on the case $0<p\leq 1$ in \Cref{rem:hardy}.}
\alert{For the domain $\Omega$, we require only the existence of an extension
operator bounded in the Besov norm.
From \cite{devore1993besov,DeVore1984}, we know this
can be ensured for $(\varepsilon, \delta)$-\emph{domains}
and for Lipschitz domains, see also \Cref{thm:extension}.}
We \alert{abbreviate}
\begin{align*}
    E(f,{\tool})_{\alert{X}}:=E(f,{\tool})_{\alert{X}}.
\end{align*}
As a measure of complexity we will use the number of non-zero weights.
\alert{For} a given $\Phi\in\repu^{r,d_1,d_2}$
for some $d_1,d_2\in\N$, the number of non-zero weights is
\begin{align*}
    W(\Phi):=\sum_{l=1}^L\norm{T_l}[\ell^0],\quad
    \norm{T_l}[\ell^0]:=\norm{A_l}[\ell^0],
\end{align*}
$\norm{A_l}[\ell^0]$ \alert{being the} number of non-zero weights of the matrix
$A_l$.
With this we define for any $n\in\N$
\begin{align*}
    \repu^{r,d_1,d_2}_n&:=\left\{\Phi\in\repu^{r,d_1,d_2}:\;\rlz(\Phi)\in X,\;
    W(\Phi)\leq n\right\},\\
    \relu_n^{d_1,d_2}&:=\repu_n^{1,d_1,d_2}.
\end{align*}
The main result of this work then concerns
the approximation classes
\begin{align}\label{eq:appclass}
    A^\alpha_q\left(\alert{X}, \repu^{r,d,1}\right).
\end{align}

\subsection{Main Result}\label{sec:main}
\alert{We summarize results on approximation of isotropic Besov spaces
-- including this work -- in \Cref{tab:result}.}

\renewcommand\arraystretch{1.5}
\begin{table}[ht]
\centering
\begin{tabularx}{\linewidth}[t]{p{1.8cm} p{1.6cm} p{2.9cm} c c c}
\toprule
$X$ & Domain $\Omega$ & Smoothness Class &
\multicolumn{2}{c}{Approximation Rate} & Reference \\
& & & $\repu^{r,d,1}$ & $\relu^{d,1}$ & \\
\midrule
$L^p(\Omega)$\newline
$1<p\leq\infty$ & $(\varepsilon, \delta)$ or\newline
Lipschitz & $B^\alpha_\tau(L^\tau(\Omega))$\newline
$\alpha\geq 1/\tau-1/p$ & $\alpha/d$ & $\sim\alpha/d$ & this work \\
$L^p(\Omega)$\newline
$0<p\leq \infty$ & $[0,1]^d$ & $B^\alpha_\tau(L^\tau(\Omega))$\newline
$\alpha>1/\tau-1/p$ & -- & $\sim\alpha/d$ & \cite{suzuki2018adaptivity}\\
$L^p(\Omega)$\newline
$0<p\leq\infty$ & bounded\newline
Lipschitz & $B^\alpha_p(L^p(\Omega))$\newline
$\alpha<r+\min(1,1/p)$ & $\alpha/d$ & $\alpha/d$ & \cite{gribonval:hal-02117139} \\
\bottomrule
\end{tabularx}
\caption{\alert{Summary of approximation rates for isotropic Besov spaces
with deep rectifier networks.
Lipschitz refers to the strong locally Lipschitz condition
from \Cref{def:Lipschitz}.
We use $\sim\alpha/d$ to indicate algebraic rates with
an additional log term or,
in other words,
any rate strictly less than $\alpha/d$ can be achieved.}}
\label{tab:result}
\end{table}

\alert{The precise statements for the direct estimates can be found in
\Cref{lemma:direct}. These estimates imply a range of interpolated
smoothness spaces \alertt{that} are continuously embedded in the approximation classes
from \eqref{eq:appclass}
\begin{align*}
            (L^p(\Omega), B^\alpha_q(L^\tau(\Omega)))_{\theta/\alpha,\bar q}
            &\hookrightarrow
            A^{\theta/d}_{\bar q}(L^p(\Omega),\repu^{r,d,1}),
\end{align*}
see the \Cref{thm:embedd} and definitions in \Cref{sec:besov}.}

\alert{The required depth to achieve the optimal rates
from \Cref{tab:result} for RePU networks with
$r\geq 2$ scales at most logarithmically in the smoothness order $\alpha$
and, in particular, is independent of the approximation error.
For ReLU networks, the required depth scales at most logarithmically
with the approximation error and at most \alertt{log-}linearly with
smoothness order $\alpha$.}
\section{Preliminaries on
\alert{Rectifier Network} Approximation}\label{sec:prelim}
In this section, we review recent results on deep RePU
approximation relevant for this work.
We use the notation defined in \Cref{sec:nn}.
The next theorem states
that RePU networks can efficiently reproduce or
approximate multiplication.

\begin{theorem}[Multiplication \cite{YAROTSKY2017103, OSZ19_839, gribonval:hal-02117139}]\label{thm:mult}
    Let $M_d:\R^d\rightarrow\R$ be the multiplication function
    $x\mapsto\prod_{i=1}^d x_i$. Then,
    there exists a constant $C$ such that
    \begin{enumerate}[label=(\roman*)]
        \item    for $r\geq 2$, and $n:=Cd$, there exists a RePU
                    network $\Phi_M\in\repu_{n}^{r, d, 1}$ such that
                    \begin{align*}
                        M_d=\rlz(\Phi_M),
                    \end{align*}
                    \alert{where the depth of $\Phi_M$ depends at most logarithmically
                    on $d$.}
        \item    \alert{For} $r=1$, any $K>0$ and any \alert{$0<\varepsilon<1$}, and
                    $n:=Cd\log(dK^d/\varepsilon)$,
                    there exists a ReLU network $\Phi_M^\varepsilon\in\relu^{d, 1}_n$ with
                    \begin{align*}
                        \norm{M_d-\rlz(\Phi_M^\varepsilon)}[{L^\infty([-K,K]^d)}]\leq\varepsilon,
                    \end{align*}
                    \alert{where the depth is at most a constant multiple of
                    $(1+\log(d)\log(dK^d/\varepsilon))$.}
    \end{enumerate}
\end{theorem}

This in turn implies RePU networks can efficiently
reproduce or approximate \alertt{piecewise} polynomials.

\begin{theorem}[\alertt{Piecewise} Polynomials \cite{YAROTSKY2017103, SchwabFEM,
gribonval:hal-02117139}]\label{thm:relupoly}
    Let $v:\R\rightarrow\R$ be a \alertt{piecewise} polynomial
    with $N_v$ pieces, of maximum degree $t\in\N_{\geq 0}$ and
    with compact support of measure $S:=|\supp(v)|<\infty$.
    Then, there exists a constant $C>0$ depending on
    $N_v$, $t$ and $r$ such that
    \begin{enumerate}[label=(\roman*)]
        \item    for $r\geq 2$ \alert{and any $\varepsilon>0$},
        there exists a RePU network $\Phi\in\repu^{r, 1, 1}_C$
        with
        \begin{align*}
             \norm{v-\rlz(\Phi)}[L^\infty(\R)]\leq\varepsilon,
        \end{align*}
        \alert{where the complexity of the network $\Phi$ is
        independent of $\varepsilon$,
        \alertt{the depth is of the order $\mc O(\log(t))$
        and $C=\mc O(N_v\log(t))$.}}
        \item    \alert{For} $r=1$, the constant $C$ additionally depends on $S$
        and $\norm{v}[L^\infty(\R)]$,
        and for any \alert{$0<\varepsilon<1$}
        there exists a ReLU network $\Phi\in\relu^{1, 1}_n$ with
        $n:=\alert{C(1+\log(\varepsilon^{-1}))}$
        and the same support as $v$, such that
        \begin{align*}
            \norm{v-\rlz(\Phi)}[L^\infty(\R)]\leq\varepsilon,
        \end{align*}
        \alert{where the depth of the network is at most
        \alertt{of the order
        $\mc O(t\log(t)\log(\varepsilon^{-1}))$
        and $C=\mc O(N_vt\log(t))$.}
        For the special case $t=1$, i.e., if $v$ is \alertt{piecewise} linear, it
        can be reproduced exactly with $v=\rlz(\Phi)$
        such that $\Phi\in\relu^{1, 1}_{C(1+N_v)}$ and has depth two.}
    \end{enumerate}
\end{theorem}


The previous result states that RePU networks with $r\geq 2$
can reproduce piece-wise polynomials of any degree at the same
asymptotic cost\footnote{Note that the constants will be, however,
affected by the degree.}.
This suggests the following \emph{saturation property}.

\begin{theorem}[Saturation Property \cite{gribonval:hal-02117139}]
    For any $r\geq 2$, $\alpha>0$, any $d_1,d_2\in\N$,
    \alert{$0<p\leq\infty$ and $\Omega\subset\R^d$
    a Borel-measurable open set with nonzero measure,}
    the approximation spaces defined in \Cref{sec:app} coincide
    \begin{align*}
        A^\alpha_q(\alert{L^p(\Omega)}, \repu^{2, d_1, d_2})=A^\alpha_q(\alert{L^p(\Omega)}, \repu^{r,d_1,d_2}).
    \end{align*}
\end{theorem}
The saturation property will also be clearly visible in the main
result of this work in \Cref{thm:embedd}.

We conclude by pointing out that RePU networks
can efficiently reproduce
\emph{affine systems}, i.e., linear combinations of functions
that are generated by dilating and shifting
a single \emph{mother} function or, in some cases, a finite
number of mother functions. A prominent example of affine systems
are wavelets which will play an important role for the main result
in \Cref{thm:embedd}.

The reproduction of affine systems by NNs was studied in greater
detail in \cite{AffineOptimal}.
In the following we only mention the properties relevant
for this work.

\begin{theorem}[NN calculus \cite{gribonval:hal-02117139}]\label{thm:calc}
    For any $r\geq 1$, the following properties hold.
    \begin{enumerate}[label=(\roman*)]
        \item\label{it:nn1}    For any $c\in\R$, $n\in\N$,
                    $d_1,d_2\in\N$ and any $\Phi_1\in\repu^{r,d_1,d_2}_n$,
                    there exists $\Phi_2\in\repu^{r,d_1,d_2}_n$ with
                    \begin{align*}
                        c\rlz(\Phi_1)=\rlz(\Phi_2),
                    \end{align*}
                   \alert{where both $\Phi_1$ and $\Phi_2$ have the same depth.} 
        \item\label{it:nn2}    For any $n_1,\ldots,n_N\in\N$, $d_1,d_2\in\N$ and any
                    $\Phi_1\in\repu^{r,d_1,d_2}_{n_1},
                    \ldots,\Phi_N\in\repu^{r,d_1,d_2}_{n_N}$,
                    set $C:=
                    \min\{d_1,d_2\}(\max_i\depth(\Phi_i)-\min_i\depth(\Phi_i))$.
                    Then, for $n:=C+\sum_{i=1}^Nn_i$,
                    there exists $\Phi_{\sum}\in\repu^{r,d_1,d_2}_n$ with
                    \begin{align*}
                        \sum_{i=1}^N\rlz(\Phi_i)=\rlz(\Phi_{\sum}),
                    \end{align*}
                    \alert{where the depth of $\Phi_{\sum}$ is bounded
                    by the maximal depth of all $\Phi_i$'s.}
        \item\label{it:nn3}     For any $n_1,\ldots,n_N\in\N$, $d,d_1,\ldots,d_N\in\N$ and any
                    $$\Phi_1\in\repu^{r,d,d_1}_{n_1},
                    \ldots,\Phi_N\in\repu^{r,d,d_N}_{n_N},$$
                    set $K:=\sum_{i=1}^Nd_i$
                    and $C:=
                    \min\{d,K-1\}(\max_i\depth(\Phi_i)-\min_i\depth(\Phi_i))$.
                    Then, for $n:=C+\sum_{i=1}^Nn_i$,
                    there exists $\Phi_{\times}\in\repu^{r,d,K}_n$ with
                    \begin{align*}
                        (\rlz(\Phi_1),\ldots,\rlz(\Phi_N))=\rlz(\Phi_{\times}),
                    \end{align*}\\
                    \alert{where the depth of $\Phi_{\times}$ is bounded
                    by the maximal depth of all $\Phi_i$'s.}
        \item\label{it:nn4}     For any $n_1,n_2\in\N$,
                    $d_1,d_2,d_3\in\N$,
                    any $\Phi_1\in\repu^{r,d_1,d_2}_{n_1}$ and any
                    $\Phi_2\in\repu^{r,d_2,d_3}_{n_2}$,
                    there exists $\Phi\in\repu^{r,d_1,d_3}_{n_1+n_2}$
                    such that
                    \begin{align*}
                        \rlz(\Phi_2)\circ\rlz(\Phi_1)=\rlz(\Phi),
                    \end{align*}
                    \alert{where the depth of $\Phi$ is simply the sum of the depths of
                    $\Phi_1$ and $\Phi_2$.}
        \item\label{it:nn5}     Let $D_b^a:\R^d\rightarrow\R^d$ denote the affine transformation
                    $x\mapsto ax-b$ for $a\in\R$, $b\in\R^d$. Then,
                    for any $n\in\N$, $d_1,d_2\in\N$, any
                    $\Phi_1\in\repu^{r, d, 1}_n$ and any $a\in\R$, $b\in\R^d$,
                    there exists $\Phi_2\in\repu^{r, d, 1}_n$ with
                    \begin{align*}
                        \rlz(\Phi_1)\circ D_b^a=\rlz(\Phi_2),
                    \end{align*}                                  
                    \alert{where both $\Phi_1$ and $\Phi_2$ are of the same depth.}
    \end{enumerate}
\end{theorem}
\section{Besov Spaces and Wavelet Systems}\label{sec:wavs}
In this section, we recall some classical results on
(isotropic) Besov spaces and
their characterization with wavelets. As in \Cref{sec:prelim},
we focus mostly on results relevant to our work.
For more details we refer to, e.g., \cite{Cohen2003}.

\subsection{Besov Spaces}\label{sec:besov}
Let $\Omega\subset\R^d$, $h\in\R^d$ \alert{and $\tau_h$ a} translation operator\\
$(\tau_h f)(x):=f(x+h)$, $\id:\R^d\rightarrow\R^d$ the identity operator
and define the \emph{$m$-th difference}
\begin{align*}
    \Delta_h^m:=(\tau_h-\id)^m:=
    \underbrace{(\tau_h-\id)\circ\ldots\circ(\tau_h-\id)}_{m\text{ times}},\quad m\in\N.
\end{align*}  
We use the notation
\begin{align*}
    \Delta_h^m (f, x, \Omega):=
    \begin{cases}
        (\Delta_h^m f)(x),&\quad\text{if } x, x+h,\ldots,x+rh\in\Omega,\\
        0,&\quad\text{otherwise}.
    \end{cases}
\end{align*}
The \emph{modulus of smoothness} of order $m$ is defined
for any $t>0$ as
\begin{align*}
    \omega_m(f, t, \Omega)_{\alert{X}}=
    \sup_{|h|\leq t}\norm{\Delta_h^m(f, \cdot, \Omega)}[\alert{X}],
\end{align*}
where $|h|$ denotes the standard Euclidean $2$-norm.
Finally, the Besov \alert{quasi-}semi-norm is defined for any $0<\alert{q}\leq\infty$,
any $\alpha>0$ and $m:=\lfloor \alpha\rfloor +1$ by
\begin{align*}
    \snorm{f}[B^\alpha_q(\alert{X})]:=
    \begin{cases}
        \left(\int_0^1[t^{-\alpha}\omega_m(f, t, \Omega)_{\alert{X}}]^q\d t/t\right)^{1/q},&
        \quad 0<q<\infty,\\
        \sup_{t>0}t^{-\alpha}\omega_m(f, t, \Omega)_{\alert{X}},&\quad q=\infty.
    \end{cases}
\end{align*}
Then, the (isotropic) Besov space is defined as
\begin{align*}
    B^\alpha_q(\alert{X}):=\left\{
    f\in \alert{X}:\; \snorm{f}[B^\alpha_q(\alert{X})]<\infty\right\}.
\end{align*}
and it is a (quasi-)Banach space \alert{that we equip with the (quasi-)norm
\begin{alignat*}{2}
    \norm{f}[B^\alpha_q(\alert{X})]&:=\norm{f}[\alert{X}]+
    \snorm{f}[B^\alpha_q(\alert{X})],\quad &&X=L^p(\Omega),\\
    \norm{f}[B^\alpha_q(\alert{X})]&:=\snorm{f}[B^\alpha_q(\alert{X})],
    \quad &&X=H_p(\R^d),
\end{alignat*}
where $H_p(\R^d)$ is the real Hardy space.}

The parameter $\alpha>0$ is the smoothness order, while the
\alert{space $X$ reflects the measure of said smoothness.}
The secondary parameter $q$ is less important and merely provides
a finer gradation of smoothness.
\alert{For $X=L^p(\Omega)$, a} few relationships are rather
\alertt{straightforward}
\begin{alignat*}{2}
    B^{\alpha_1}_q(L^p(\Omega))&\hookrightarrow
    B^{\alpha_2}_q(L^p(\Omega)),&&\quad \alpha_1\geq\alpha_2,\\
    B^{\alpha}_q(L^{p_1}(\Omega))&\hookrightarrow
    B^{\alpha}_q(L^{p_2}(\Omega)),&&\quad p_1\geq p_2,\\
    B^{\alpha}_{q_1}(L^p(\Omega))&\hookrightarrow
    B^{\alpha}_{q_2}(L^p(\Omega)),&&\quad q_1\leq q_2,
\end{alignat*}
where $\hookrightarrow$ denotes a continuous embedding.
For non-integer $\alpha>0$ and $1\leq p\leq\infty$,
$B^{\alpha}_p(L^p(\Omega))$ is the fractional Sobolev space
$W^{\alpha,p}(\Omega)$. For integer $\alpha>0$,
the Besov space $B^{\alpha}_\infty(L^p(\Omega))$
is slightly larger than $W^{\alpha,p}(\Omega)$.
For $p=q=2$, the Besov space $B^{\alpha}_2(L^2(\Omega))$
is the same as the Sobolev space $W^{\alpha,2}(\Omega)$.

\alert{
The Besov spaces
$B_\tau^\alpha(L^\tau(\Omega))$, $1/\tau=\alpha/d+1/p$}
are on
the \emph{critical embedding line} (see \Cref{fig:DeVore}).
Spaces above this line are embedded in $L^p$,
spaces on this line may or may not be embedded in $L^p$,
and spaces below this line are never embedded in $L^p$.
In this sense, such Besov spaces are quite large as the functions
on this line barely have enough regularity to be members of
$L^p$. It is well-known that optimal approximation
\alert{-- in the sense of continuous nonlinear widths
(see also beginning of \Cref{sec:intro}) --}
of functions from such spaces with
a continuous parameter selection
can only be achieved by nonlinear methods,
see \cite{devore1989optimal}.
It is the main result of this work that
RePU networks achieve optimal approximation for these spaces,
while ReLU networks achieve near to optimal approximation.

To transfer results from $\R^d$ to \alert{more general domains},
we will use the common technique of extension operators.
\begin{theorem}[Extension Operator \cite{DeVore1984, devore1993besov}]\label{thm:extension}
    \alert{
    Let $\alpha>0$, $0<p,q\leq\infty$
    and let $\Omega\subset\R^d$ be an $(\varepsilon,\delta)$-domain
    for $0<p\leq 1$ and a strong locally Lipschitz domain for
    $1<p\leq\infty$.
    Then,}
    there exists an \alert{extension} operator
    $\mc E:B^\alpha_q(L^p(\Omega))\rightarrow
    B^\alpha_q(L^p(\R^d))$ such that
    \alert{$\mc Ef|_{\Omega}=f$ and}
    \begin{align*}
        \norm{f}[B^\alpha_q(L^p(\Omega))]\leq
        \norm{\mc Ef}[B^\alpha_q(L^p(\R^d))]\leq C
        \norm{f}[B^\alpha_q(L^p(\Omega))],
    \end{align*}
    where $C$ depends only on $d$, $\alpha$, $p$ and
    the domain $\Omega$.
\end{theorem}

We conclude by noting that Besov spaces combine well with
interpolation. To be precise, we briefly define
interpolation spaces via the $K$-functional.
Let $X$ be a quasi-normed space
and $Y$ be a quasi-semi-normed space
with $Y\hookrightarrow X$.
The $K$-functional is defined for any $f\in X$ by
\begin{align*}
    K(f, t, X, Y):=\inf_{f=f_0+f_1}\{\norm{f_0}[X]+t\snorm{f_1}[Y]\},\quad t>0.
\end{align*}
For $0<\theta<1$ and $0<q\leq\infty$, define the quantity
\begin{align*}
    \snorm{f}[(X,Y)_{\theta,q}]:=
    \begin{cases}
        \left(\int_0^\infty
        [t^{-\theta}K(f,t,X,Y)]^q\d t/t\right)^{1/q},&\quad 0<q<\infty,\\
        \sup_{t>0}t^{-\theta} K(f,t,X,Y),&\quad q=\infty.
    \end{cases}
\end{align*}
Then, the spaces
\begin{align*}
    (X,Y)_{\theta,q}:=\left\{
    f\in X:\snorm{f}[(X,Y)_{\theta,q}]<\infty\right\},
\end{align*}
equipped with the (quasi-)norm
\begin{align*}
    \norm{f}[(X,Y)_{\theta,q}]:=\norm{f}[X]+\snorm{f}[(X,Y)_{\theta,q}],
\end{align*}
are \emph{interpolation spaces}.

Besov spaces provide a relatively complete description
of interpolation spaces in the following sense:
for $0<\theta<1$
\begin{alignat*}{2}
    (L^p(\Omega), W^\alpha(L^p(\Omega)))_{\theta,q}&=
    B^{\theta\alpha}_q(L^p(\Omega)),&&\quad
    1\leq p\leq\infty,\; 0<q\leq\infty,\\
    (B^{\alpha_1}_{q_1}(L^p(\Omega)),
    B^{\alpha_2}_{q_2}(L^p(\Omega)))_{\theta,q}&=
    B^{\alpha}_q(L^p(\Omega)),&&\quad
    0<\alpha_1<\alpha_2,\; \alpha:=(1-\theta)\alpha_1+\theta\alpha_2,\\
    & &&\quad 0<p,q,q_1,q_2\leq\infty,\\
    (L^p(\Omega),
    B^{\alpha}_{q_1}(L^p(\Omega)))_{\theta,q}&=
    B^{\theta\alpha}_q(L^p(\Omega)),&&\quad
    0<p,q,q_1\leq\infty.
\end{alignat*}
For Besov spaces on the critical line with $1/\tau=\alpha/d+1/p$,
we obtain
\begin{align*}
    (L^p(\Omega),
    B^{\alpha}_{\tau}(L^\tau(\Omega)))_{\theta,q}&=
    B^{\theta\alpha}_q(L^q(\Omega)),\quad
    \text{if}\quad 1/q=\theta\alpha/d+1/p.
\end{align*}

\subsection{Wavelets}
There are many possible wavelets constructions satisfying different properties
depending on the intended application.
Said constructions can be rather technical, with the payoff being
various favorable analytical and numerical features.
We do not intend to cover this topic in-depth and once again only pick out
the aspects required for this work.
We proceed by briefly reviewing one-dimensional wavelets constructions,
after which we turn to wavelets on $\R^d$.
Our presentation is somewhat abstract and therefore flexible,
but we will also be more specific with some aspects
of the construction that we require in \Cref{sec:optimal}.
For more details on the
subject we refer to \cite{Cohen2003}.

The starting point of a wavelet construction is typically a
multi-resolution analysis (MRA), i.e., a sequence of closed
subspaces $V_j\subset V_{j+1}$ of $L^2(\R)$ that
are nested, dilation- and shift-invariant, dense in $L^2(\R)$
and are all generated by a single\footnote{Multiple
scaling functions are possible as well in which
case such functions are referred to
as \emph{multi-wavelets}, see \cite{donovan}.} \emph{scaling function}
$\varphi\in V_0$. To be more precise,
we assume the system
$\{\varphi(\cdot-k):\;k\in\Z\}$
is a Riesz basis of $V_0$ and therefore
$\{\varphi(2^j\cdot-k):\;k\in\Z\}$ is a Riesz basis of $V_j$.
We use the shorthand notation
\begin{align}\label{eq:defdil}
    \varphi_{j,k}:=2^{j/2}\varphi(2^j\cdot-k),
\end{align}
where the pre-factor $2^{j/2}$ normalizes $\varphi$ in $L^2$.
Later we will redefine this to $2^{j/p}$ for normalization in
$L^p$ \alert{or $H_p$} for any $0<p\leq\infty$, with
the convention $2^{j/\infty}=1$.
\alert{Here $H_p$ denotes the real Hardy space which coincides with
$L^p$ for $p>1$, see \cite{Fefferman1972} and \Cref{rem:hardy}.}

Defining a projection $P_j:L^2(\R)\rightarrow V_j$ is rather simple if $\{\varphi(\cdot-k):\;k\in\Z\}$ forms an orthogonal basis of $V_0 $. Indeed, this   property implies that $\{\varphi_{j,k}:\;k\in\Z\}$ forms an orthogonal basis of $V_j$, and $P_j$ can be chosen to be the orthogonal projection.
However, for numerical reasons, it is sometimes unpractical
to construct scaling functions $\varphi$ such that $\{\varphi(\cdot-k):\;k\in\Z\}$ forms an orthogonal basis of $V_0 $ and without this property a constructive definition of $P_j$ is not \alertt{straightforward}.

A way-out are so-called \emph{bi-orthogonal} constructions.
A function $\tilde\varphi\in L^2(\R)$ is dual to $\varphi$ if it satisfies
\begin{align*}
    \langle\varphi(\cdot-k),\tilde\varphi(\cdot-l)\rangle_{L^2}
    =\delta_{k,l},\quad k,l\in\Z,
\end{align*}
where $\delta_{k,l}$ is the Kronecker delta.
We then define the \emph{oblique} projection $P_j$
\begin{align*}\label{eq:oblique}
    P_j f:=\sum_{k\in\Z}\langle f,\tilde\varphi_{j,k}\rangle_{L^2}
    \varphi_{j,k}.
\end{align*}

A representation of a function in $V_j$ is typically referred to
as a \emph{single-scale} representation. To switch to a
\emph{multi-scale} representation,
we need to characterize the so-called \emph{detail spaces}
defined through the projections
\begin{align*}
    Q_j:=P_{j+1}-P_j,
\end{align*}
with the detail spaces defined as $W_j:=Q_j(L^2(\R))$.
This is achieved by constructing a \emph{wavelet} $\psi\in V_1$
\begin{align*}
    \psi:=\sum_{k\in\Z}g_k\varphi(2\cdot-k),
\end{align*}
for some coefficients $g_k\in\R$ such that
\begin{align*}\label{eq:Npsi}
    N_\psi:=\#\{k:g_k\neq 0\}<\infty.
\end{align*}

Any function $f\in L^2(\R)$ can then be decomposed into
a sequence of \alert{detail coefficients}
\begin{align}\label{eq:decompinit}
    f=\alert{\sum_{j\in\Z}\sum_{k\in\Z}c_{j,k}\psi_{j,k},}
\end{align}
\alert{where $\psi_{j,k}$ is defined as in
\eqref{eq:defdil}.}
To simplify notation, one typically \alert{introduces
the index set $\nabla:=\Z\times\Z$.}
Decomposition \eqref{eq:decompinit} then simplifies to
\begin{align*}
    f=\sum_{\lambda\in\nabla}c_\lambda\psi_\lambda.
\end{align*}

In order for the wavelets $\psi_\lambda$ to characterize
Besov spaces, they have to
satisfy certain assumptions.
\begin{assumption}[Characterization]
    We assume the scaling function $\varphi$ and its dual
    $\tilde\varphi$ satisfy the following properties.
    \begin{enumerate}[label=(W\arabic*)]
        \item\label{W1} (Integrability) For some $p', p''\in [1,\infty]$
        such that $1/p'+1/p''=1$, we
        assume $\varphi\in L^{p'}(\R)$
        and $\tilde\varphi\in L^{p''}(\R)$. 
        \item\label{W2} (Polynomial Reproduction)
        We assume $\varphi$ satisfies \emph{Strang-Fix} conditions
        of order $L\in\N$ or, equivalently, for any polynomial $P\in\mb P_{L-1}$
        of degree $L-1$, we have $P\in V_0$.
        \item\label{W3} (Regularity)   For some
        $s>0$ \alert{and some $0<\tau,q\leq\infty$,
        we assume $\varphi\in B^s_{q}(L^\tau(\R))$}.
    \end{enumerate}
\end{assumption}

These conditions are sufficient to ensure Besov spaces can be characterized
by the decay of the wavelet coefficients \alert{for the case $p\geq 1$
or given sufficient regularity}.
\alert{For the case $X=H_p(\R^d)$ and $0<p\leq 1$, we additionally
require $\varphi$ to satisfy}\footnote{\alert{Weaker assumptions are also possible,
see \cite{Kyriazis1996Wavelet} for details.}}

\begin{assumption}[\alert{Hardy Spaces}]
        \alert{For $0<p\leq 1$, assume $\varphi$ satisfies}
        \begin{enumerate}[label=(H\arabic*)]
        \item\label{H1} 
        \alert{$\hat{\varphi}(0)=1$, $D^\beta\hat{\varphi}(0)=0$
        for every $1\leq|\beta|<\max\{[d(1/p-1)]+1, L\}$,
        where $\hat{\varphi}$ denotes the Fourier transform
        of $\varphi$.}
        \end{enumerate}
\end{assumption}

\alert{For $N$-term approximation results in the
case $p=\infty$, we additionally require}

\begin{assumption}[\alert{$p=\infty$}]\leavevmode
    \begin{enumerate}[label=(A\arabic*)]
        \item\label{P1}    \alert{We assume $\varphi$ has compact support
                                    and is $s$ times continuously differentiable.}
    \end{enumerate}
\end{assumption}

\alert{Finally,} for our work we will require two
additional conditions that are, however, easy to satisfy
for a variety of wavelet families.

\begin{assumption}[\alertt{Piecewise} Polynomial]
    We additionally assume the scaling function $\varphi$ satisfies the following
    properties.
    \begin{enumerate}[label=(P\arabic*)]
        \item\label{A1}    We assume $\varphi$ has compact support.
        \item\label{A2}    We assume $\varphi$ is \alertt{piecewise} polynomial.
    \end{enumerate}
\end{assumption}

An example of \alert{wavelet families that can be
constructed to satisfy} all
of the assumptions
\alert{\ref{W1}--\ref{W3}, \ref{H1}, \ref{P1} and \ref{A1}--\ref{A2}}
are the CDF bi-orthogonal B-spline wavelets from \cite{cdf}.
These constructions allow to choose an arbitrary polynomial reproduction
degree $L-1$, regularity order $s$ and the resulting scaling function
$\varphi$ (and consequently $\psi$ as well) are compactly
supported splines of degree $L-1$.

Finally, we briefly describe how to extend the above wavelets
to \alert{the multivariate case}.
There are several possible approaches for this,
but we describe a specific tensor product construction suitable
for isotropic Besov spaces.

For $x\in\R^d$, we define the tensor product scaling function as
\begin{align*}\label{eq:tpscaling}
    \phi(x):=\varphi(x_1)\cdots\varphi(x_d),
\end{align*}
and in the same manner as before, but for a general $0<p\leq\infty$,
\begin{align}\label{eq:tpscaling2}
    \phi_{j,k,p}(x):=2^{dj/p}\phi(2^jx-k),\quad j\in\Z,\;k\in\Z^d,
\end{align}
with the convention $2^{dj/\infty}=1$.
Next, for $e\in\{0,1\}^d\setminus\{0\}$, we define
\begin{align}\label{eq:tpwavelet}
    \psi^e(x):=\psi^{e_1}(x_1)\cdots\psi^{e_d}(x_d),
\end{align}
with the convention $\psi^1(x_i):=\psi(x_i)$ and $\psi^0(x_i)=\varphi(x_i)$,
and $\psi^e_{j,k,p}$ is defined as in \eqref{eq:tpscaling2}.
Simplifying as before with
\begin{align*}
    \alert{\nabla:=\{(e,j,k):\;e\in\{0,1\}^d\setminus\{0\},\;j\in\Z,\;k\in\Z^d\}},
\end{align*}
we obtain the $d$-dimensional wavelet system
\begin{align}\label{eq:defwavsys}
    \Psi:=\{\psi_{\lambda,p}:\;\lambda\in\nabla\}.
\end{align}
Finally, we define the fixed level sets
\begin{align*}
    \alert{\nabla_j:=\{\lambda=(e,j,k)\in\nabla:\;|\lambda|=j\},}
\end{align*}
\alert{where we use the shorthand
notation $|\lambda|:=|(e,j,k)|:=j$.}

\begin{theorem}[Characterization \cite{Cohen2003,Kyriazis1996Wavelet}]\label{thm:char}
    Let $\varphi$ satisfy \ref{W1} for some integrability parameters
    $p',p''$, \ref{W2} for order $L$ and \ref{W3} with smoothness
    order $s$ for primary parameter $0<p\leq p'$ and any secondary parameter
    $0<q\leq\infty$.
    
    \alert{For $\alpha<\min\{s,L\}$, assume additionally either}
    \begin{enumerate}[label=(\roman*)]
        \item\label{it:suffreg} \alert{$X=L^p(\R^d)$, $0<p\leq\infty$ and
        $\alpha>d(1/p-1/p')$},
        \item \alert{or $X=H_p(\R^d)$, $0<p\leq 1$
        and \ref{H1}.}
        \item \alertt{or $X=L^p(\R^d)$, $1\leq p\leq \infty$,
        $\varphi$ bounded and compactly supported.}   
    \end{enumerate}

    Then, if $f=\sum_{\lambda\in\nabla}c_{\lambda,p}\psi_{\lambda,p}$
    is the wavelet decomposition of $f$,  
    we have the norm equivalence
    \begin{align*}
        \snorm{f}[B^\alpha_q(\alert{X})]\sim
        \begin{cases} 
            \left(\sum_{\alert{j\in\Z}}2^{j\alpha q}
            \left(\sum_{\lambda\in\nabla_j}
            |c_{\lambda,p}|^p\right)^{q/p}\right)^{1/q},
            &\quad 0<q<\infty,\\
            \sup_{\alert{j\in\Z}}2^{j\alpha}
            \left(\sum_{\lambda\in\nabla_j}
            |c_{\lambda,p}|^p\right)^{1/p},
            &\quad q=\infty.
        \end{cases}
    \end{align*}
\end{theorem}

\alertt{
Note the renormalization relationship
\begin{align}\label{eq:renorm}
	c_{\lambda,p}=2^{-|\lambda|d(1/p-1/q)}c_{\lambda,q},\quad
	0<p,q\leq\infty
\end{align}
and in particular
$c_{\lambda,p}=2^{-|\lambda|d(1/p-1/2)}c_{\lambda,2}$
with $c_{\lambda,2}=\langle f,\tilde\psi_{\lambda,2}\rangle_{L^2}$,
the inner product with the $L^2$-scaled dual wavelet
$\tilde\psi_{\lambda,2}$.
}

The above characterization implies \alert{the following}
approximation rates for best $N$-term wavelet approximations.
\begin{theorem}[$N$-term Approximation \cite{Cohen2003,Kyriazis1996Wavelet,devorec}]\label{thm:Nterm}
    Let $\alert{0}<p\alert{\leq}\infty$ and let $\Psi$ be a wavelet system satisfying
    \alert{\ref{W1}--\ref{W3} with regularity $s$ for Besov primary parameter
    $\tau>0$, reproduction order $L$,
    \alertt{$0<\alpha<\min\{s,L\}$ and
    assume either}}
    \begin{enumerate}[label=(\roman*)]
        \item \alert{$X=L^p(\R^d)$, $1<p<\infty$}
        \alertt{and $d(1/p-1/p')<\alpha$},
        \item \alertt{or $X=L^p(\R^d)$, $1<p<\infty$,
        $\varphi$ is bounded and compactly supported},
        \item \alert{or $X=L^\infty(\R^d)$, \ref{P1}
        and $d\leq \alpha\leq\min\{s,L\}$},
        \item \alert{or $X=H_p(\R^d)$, $0<p\leq 1$
        and \ref{H1}.}   
    \end{enumerate}
    
    Define the set of $N$-term wavelet expansions as
    \begin{align*}
        \mc W_N:=\left\{\sum_{\lambda\in\Lambda}c_{\lambda,p}
        \psi_{\lambda,p}:\;\Lambda\subset\nabla,\;\#\Lambda\leq N\right\}.
    \end{align*}        
    Then, \alert{for $\alpha/d\geq 1/\tau-1/p$, $0<q\leq\tau\leq p$
    and any $f\in B^\alpha_{q}(Y^\tau)$, it holds}
    \begin{align*}
        E(f, \mc W_N)_{\alert{X}}\lesssim N^{-\alpha/d}
        \snorm{f}[\alert{B^\alpha_{q}(Y^\tau)}],
    \end{align*}
    \alert{where $Y^\tau:=H_\tau(\R^d)$ for $0<p\leq 1$,
    $Y^\tau:=L^\tau(\R^d)$ for $1< p\leq\infty$
    and for $p=\infty$ we additionally assume $f$ is continuous.}
\end{theorem}
\section{Optimal ReLU Approximation of Smoothness Classes}\label{sec:optimal}
With the results from \Cref{sec:prelim} and \Cref{sec:wavs} we have
all the tools necessary to derive approximation rates
for arbitrary Besov functions. As was reviewed in \Cref{sec:wavs},
Besov spaces can be characterized by the decay of the wavelet
coefficients, and $N$-term approximations achieve optimal
approximation rates for Besov functions.

In this section, we show that a RePU network
can reproduce an $N$-term wavelet expansion with $\mc O(N)$
complexity. More importantly,
we also show that a ReLU network
can approximate an $N$-term wavelet
expansion with $\mc O(N\log(\varepsilon^{-1}))$ complexity,
where $\varepsilon>0$ is the related approximation accuracy.
Together with a stability estimate,
this will imply RePU and ReLU networks can approximate
Besov functions of arbitrary smoothness order with optimal
or near to optimal rate, respectively.

\begin{lemma}[Wavelet System Complexity]\label{lemma:wavsd}
    Let $\Psi$ be a wavelet system as defined in \eqref{eq:defwavsys},
    with \alert{the}
    one-dimensional scaling function $\varphi:\R\rightarrow\R$
    satisfying \ref{A1} -- \ref{A2}. Then,
    \alert{for $X=L^p(\R^d)$, $0< p\leq\infty$},
    \begin{enumerate}[label=(\roman*)]
        \item    for $r\geq 2$, there exists a constant $C>0$
                    \alert{depending on $r$, polynomial reproduction order $L$,
                    dimension $d$, support of
                    $\varphi$, $\psi$ and
                    $\norm{\varphi}[L^\infty(\R)]$, $\norm{\psi}[L^\infty(\R)]$},
                    such that for any
                    $\psi_\lambda\in\Psi$ \alert{and any
                    $\varepsilon>0$,} there exists a RePU network
                    $\Phi_\lambda^\varepsilon\in\repu^{r,d,1}_C$ with
                    \alert{the same support as $\psi_\lambda$ such that}
                    \begin{align*}
                        \norm{\psi_\lambda-\Phi_\lambda^\varepsilon}[L^p(\R^d)]\leq\varepsilon,
                    \end{align*}
                    \alert{where the complexity of the network $\Phi_\lambda^\varepsilon$
                    is independent of $\varepsilon$,
                    the depth is at most logarithmic in $L$ and $d$.}
        \item\label{it:relu}    For $r=1$,
                    there exists a constant $C>0$
                    with dependencies as \alert{above},
                    such that
                    for any $\psi_\lambda\in\Psi$ and any $0<\varepsilon<1$,
                    there exists a ReLU network $\Phi_\lambda^\varepsilon
                    \in\relu^{d,1}_{C(1+\log(\varepsilon^{-1}))}$ with
                    \alert{the same support as $\psi_\lambda$ such that}
                    \begin{align*}
                        \norm{\psi_\lambda-\Phi_\lambda^\varepsilon}[L^p(\R^d)]\leq\varepsilon,
                    \end{align*}
                    \alert{where the depth of the network $\Phi_\lambda^\varepsilon$
                    is at most logarithmic in
                    \alertt{$\varepsilon^{-1}$},
                    \alertt{log-linear in $L$ and
                    log-linear in $d$.}}
    \end{enumerate}
\end{lemma}

\begin{proof}
    \alert{
    We detail the proof for the case of ReLU networks
    \ref{it:relu} as the proof for RePU networks is analogous and
    more \alertt{straightforward}.
    In the following we will frequently use the triangle inequality, i.e.,
    assuming $p\geq 1$. For $0<p<1$, $\norm{\cdot}[L^p(\R^d)]$
    is only a quasi-norm, i.e., the right-hand-side
    of the triangle inequality is to be multiplied by a constant,
    and the corresponding complexities are to be adjusted accordingly.
    }

    \alert{
    First, due to \ref{A1} -- \ref{A2}, the mother wavelet
    $\psi:\R\rightarrow\R$ is a compactly supported \alertt{piecewise} polynomial.
    By \Cref{thm:relupoly}, for any $0<\delta<1$, there exists a ReLU network
    $\Phi_\psi^\delta\in\relu^{1,1}_{C(1+\log(\delta^{-1}))}$ with the same support
    as $\psi$ such that $\|\psi-\rlz(\Phi_\psi^\delta)\|_{L^\infty(\R)}\leq\delta$.
    The depth of $\Phi_\psi^\delta$ is at most logarithmic in
    \alertt{$\delta^{-1}$} and \alertt{log-}linear
    in the polynomial degree of $\psi$.
    A similar conclusion holds for the scaling function $\varphi:\R\rightarrow\R$.
    For RePU networks, $\Phi_\psi^\delta$ has complexity
    independent of $\delta$, with depth at most logarithmic in
    \alertt{the} polynomial degree
    of $\psi$.
    }

    \alert{
    Second, recall the tensor product wavelet from \eqref{eq:tpwavelet}:
    $\psi^e(x):=\psi^{e_1}(x_1)\cdots\psi^{e_d}(x_d),$ where
    for each component either $\psi^{e_\nu}=\psi^1=\psi$
    or $\psi^{e_\nu}=\psi^0=\varphi$.
    Let $\Phi_{\psi}^\delta$, $\Phi^\delta_{\alertt{\varphi}}\in\relu^{1,1}_{C(1+\log(\delta^{-1}))}$
    be the ReLU networks as above
    \alertt{with $C=\mc O(L\log(L))$}, approximating $\psi$ and $\varphi$, respectively,
    with accuracy $\delta>0$ to be specified later.
    Then, we form the tuple
    $\Phi_\times^\delta\in\relu^{d,d}_{n_\delta}$ as in \Cref{thm:calc} \ref{it:nn3}
    \begin{align*}
        \rlz(\Phi_\times^\delta) =
        (\rlz(\Phi_{\psi^{e_1}}^\delta),\ldots,\rlz(\Phi_{\psi^{e_d}}^\delta)),
    \end{align*}
    with $n_\delta:=dC(1+\log(\delta^{-1}))$
    \alertt{and where each $\Phi_{\psi^{e_\nu}}$ has either
    the depth of $\Phi_{\varphi}$ or $\Phi_{\psi}$ --
    both of the same order as discussed above}.
    The depth of $\Phi_\times^\delta$ is then the same as the maximal depth
    between $\Phi_{\varphi}^\delta$ and $\Phi_{\psi}^\delta$ -- logarithmic
    in \alertt{$\delta^{-1}$}.
    }

    \alert{
    Next, for some different accuracy $\eta>0$,
    we construct an approximate multiplication network
    $\Phi_M^\eta\in\relu^{d,1}_{n_\eta}$ with
    $n_\eta:=Cd\log(dK^d\eta^{-1})$ as
    in \Cref{thm:mult} \ref{it:nn2}, where
    $K:=\max\{\norm{\varphi}[L^\infty(\R)],
    \norm{\psi}[L^\infty(\R)]\}+\delta$.
    This choice of $K$ is justified by
    \begin{align}\label{eq:defK}
        \norm{\rlz(\Phi^{\delta}_{\psi^{e_\nu}})}[L^\infty(\R^d)]\leq
        \delta + \norm{\psi^{e_\nu}}[L^\infty(\R)]\leq K.
    \end{align}
    Our final approximation
    $\Phi^\varepsilon_\phi\in\relu^{d,1}_{n_{\delta,\eta}}$    
    is defined by
    \begin{align}\label{eq:depthphix}
        \rlz(\Phi^\varepsilon_{\psi^e}):=\rlz(\Phi_M^\eta)\circ\rlz(\Phi_\times^\delta),
    \end{align}
    where, according to \Cref{thm:calc} \ref{it:nn4},
    $n_{\delta,\eta}=dC(1+\log(\delta^{-1}))+Cd\log(dK^d\eta^{-1})$,
    \alertt{where as above $C=\mc O(L\log(L))$.}
    The depth of $\Phi^\varepsilon_{\psi^e}$ is
	\alertt{the sum of the depths of
	$\Phi_M^\eta$ and $\Phi_\times^\delta$, i.e.,}    
    at most
    a multiple of
    $(1+\log(d)\log(dK^d\eta^{-1}) +
    \alertt{L\log(L)}\log(\delta^{-1}))$.
    }

    We estimate the resulting error from which it will be clear
    how to choose $\delta,\eta>0$ and the resulting cost $n_{\delta,\eta}$.
    We introduce the auxiliary approximation
    $\rlz(\tilde\Phi):=M_d\circ \rlz(\Phi_\times^\delta)$
    and the notation $\psi^{e_\nu}_\delta:=\rlz(\Phi_{\psi^{e_\nu}}^\delta)$.
    Then,
    \begin{align}\label{eq:totalerror}
        \norm{\psi^e-\rlz(\Phi^\varepsilon_{\psi^e})}[L^p(\R^d)]
        \leq\norm{\psi^e-\rlz(\tilde\Phi)}[L^p(\R^d)]+
        \norm{\rlz(\tilde\Phi)-\rlz(\Phi^\varepsilon_{\psi^e})}[L^p(\R^d)].
    \end{align}
    With $S:=|\supp(\varphi)\cup \supp(\psi)|$, for the second term we apply
    \Cref{thm:mult} \ref{it:nn2} and obtain
    \begin{align}\label{eq:eta}
        \norm{\rlz(\tilde\Phi)-\rlz(\Phi^\varepsilon_{\psi^e})}[L^p(\R^d)]&\leq
        \norm{\rlz(\tilde\Phi)-\rlz(\Phi^\varepsilon_{\psi^e})}[L^\infty(\R^d)]
        \left(\int_{\supp(\rlz(\tilde\Phi))\cup\supp(\rlz(\Phi^\varepsilon_{\psi^e}))}
        \right)^{1/p}\notag\\
        &\leq
        \eta S^{d/p}.    
    \end{align}
    For the first term, we can write
    \begin{align*}
        \psi^e-\rlz(\tilde\Phi)=
        &(\psi^{e_\nu}-\psi^{e_\nu}_\delta)\otimes\psi^{e_2}\otimes\cdots\otimes
        \psi^{e_d}
        +
        \psi^{e_1}_\delta\otimes(\psi^{e_2}-\psi^{e_2}_\delta)\otimes
        \psi^{e_3}_\delta\otimes\cdots\otimes\psi^{e_d}\\
        &+
        \psi^{e_1}_\delta\otimes\cdots\otimes\psi^{e_{d-1}}_\delta\otimes
        (\psi^{e_d}-\psi^{e_d}_\delta).
    \end{align*}
    Thus, \alertt{for estimating \eqref{eq:totalerror},}
    by a triangle inequality
    \begin{align}\label{eq:delta}
        \norm{\psi^e-\rlz(\tilde\Phi)}[L^p(\R^d)]\leq
        dK^{d-1}\delta=d(\max\{\norm{\varphi}[L^\infty(\R)],
    \norm{\psi}[L^\infty(\R)]\}+\delta)^{d-1}\delta.
    \end{align}
    From \eqref{eq:eta} and \eqref{eq:delta},
    we set $\eta:=S^{-d/p}\varepsilon/2$ and
    $$\alertt{
    \delta:=\varepsilon (\max\{1,\norm{\varphi}[L^\infty(\R)],
    \norm{\psi}[L^\infty(\R)]\})^{1-d}/(d2^d)<1}.$$
    \alertt{With this $\delta$, we can estimate
    \begin{align*}
    	dK^{d-1}\delta
    	&\leq
    	d(\max\{\norm{\varphi}[L^\infty(\R)],
    	\norm{\psi}[L^\infty(\R)]\})^{d-1}
    	(\delta^{1/(d-1)}+\delta^{d/(d-1)})^{d-1}\\
    	&\leq
    	d(\max\{\norm{\varphi}[L^\infty(\R)],
    	\norm{\psi}[L^\infty(\R)]\})^{d-1}
    	(2\delta^{1/(d-1)})^{d-1}\\
    	&=d(\max\{\norm{\varphi}[L^\infty(\R)],
    	\norm{\psi}[L^\infty(\R)]\})^{d-1}
    	2^{d-1}\delta
    	\leq\varepsilon/2,
    \end{align*}}
    \alertt{Thus, overall} we obtain \alertt{for \eqref{eq:totalerror}}
    \begin{align*}
        &\norm{\psi^e-\rlz(\Phi^\varepsilon_{\psi^e})}[L^p(\R^d)]\leq
        \varepsilon,
    \end{align*}
    with number of nonzero weights
    \begin{align*}
            &n_\varepsilon:=
            n_{\delta,\eta}=dC(1+\log(\delta^{-1}))+Cd\log(dK^d\eta^{-1})=\\
            &\alertt{
            \mc O
            \left(dL\log(L)[d\log(
            1+\max\{\norm{\varphi}[L^\infty], \norm{\psi}[L^\infty]\})
            +\log(\varepsilon^{-1})+d\log(S)+\log(d)]
            \right)},    
    \end{align*}        
    and,
	\alertt{according to \eqref{eq:depthphix}},    
    depth at most a multiple of
    \begin{align*}
              &1+\log(d)\log(dK^d\eta^{-1}) +
              \alertt{L\log(L)}\log(\delta^{-1})=
    \\
    &\alertt{
            \mc O
            \left(\log(d)L\log(L)[d\log(
            1+\max\{\norm{\varphi}[L^\infty], \norm{\psi}[L^\infty]\})
            +\log(\varepsilon^{-1})+d\log(S)
            +\log(d)]\right)}.        
    \end{align*}
    
    Finally, we define
    $\Phi_\lambda^\varepsilon\in\relu^{d,1}_{C(1+\log(\varepsilon^{-1}))}$ using
    \Cref{thm:calc} \ref{it:nn1} and \ref{it:nn5} such that for $\lambda=(e,j,k)$
    \begin{align*}
        \rlz(\Phi_\lambda^\varepsilon)=2^{dj/p}\rlz(\Phi^\varepsilon_{\psi^e})\circ
        D_k^{\alert{2^j}}.
    \end{align*}
    Note that the error bound
    for $\Phi_\lambda^\varepsilon$ remains unchanged due to the
    normalization constant $2^{dj/p}$.
\end{proof}

With this we turn to direct estimates for networks.
\begin{lemma}[Direct Estimates RePU/ReLU]\label{lemma:direct}
    \alert{Let $X=L^p(\Omega)$, $1<p\leq\infty$,
    $f\in B^\alpha_{q}(L^\tau(\Omega))$ with
    $\alpha,\tau,q>0$ and
    \begin{align*}
            \alpha/d\geq 1/\tau-1/p,\quad \alert{0<q\leq\tau\leq p},
    \end{align*}
    and assume $\Omega\subset\R^d$ is an $(\varepsilon,\delta)$-domain
    for $\tau\leq 1$ and a strong locally Lipschitz domain for $\tau>1$.
    For $p=\infty$, assume additionally $f$ is continuous
    with the convention $1/\infty=0$.
    Then,}
    \begin{enumerate}[label=(\roman*)]
        \item    for $r\geq 2$,
        \alertt{there exists a constant $C>0$ such that}
        \begin{align*}
            E(f,\repu^{r,d,1}_n)_{\alert{L^p(\Omega)}}\alert{\leq C} n^{-\alpha/d}
            \alert{\norm{f}[\alert{B^\alpha_{q}(L^\tau(\Omega))}]},
        \end{align*}
        \alertt{for all $n\in\N$}.
        \alert{The constant $C$ depends on $r$, $\alpha$, $\tau$, $p$, $d$ and
        $\Omega$.
        The networks realizing these approximations have depth at most
        logarithmic in $\alpha$ and $d$.}

        \item    For $r=1$,
        \alertt{there exists a constant $C>0$ such that}
        \begin{align*}
            E(f,\relu_n^{d,1})_{\alert{L^p(\Omega)}}\alert{\leq C} n^{-\bar\alpha/d}
            \alert{\norm{f}[\alert{B^\alpha_{q}(L^\tau(\Omega))}]},
        \end{align*}
        \alertt{for all $n\in\N$ and}
        for any $0<\bar\alpha<\alpha$.
        \alert{The constant $C$ depends $\alpha$, $\bar\alpha$, $\tau$, $p$, $d$ and
        $\Omega$.
        The networks realizing these approximations have depth
        at most logarithmic in $n$, \alertt{log-}linear in $\alpha$ and
        \alertt{log-linear} in $d$.}
    \end{enumerate}
\end{lemma}

\begin{proof}
    Consider a wavelet system $\Psi$ that satisfies \alert{\ref{A1} -- \ref{A2},
    \ref{W1} -- \ref{W3} and $\ref{P1}$ for $p=\infty$.
    Such a wavelet system can be constructed, e.g., as in \cite{cdf}
    and we can use $\Psi$ for $N$-term approximations as
    in \Cref{thm:Nterm}.
    We detail the proof for ReLU networks in part (ii), part (i) follows
    analogously with fewer technicalities.}
    
    Let $\Omega=\R^d$
    and $f_N:=\sum_{\lambda\in\Lambda_N}c_{\lambda,p}(f)
    \psi_\lambda$ with $\#\Lambda_N\leq N$ be
    \alert{a best $N$-term wavelet
    approximation to $f$}.
    For $\varepsilon>0$ to be specified later,
    let $\Phi_\lambda^\varepsilon\in\relu^{d,1}_{C(1+\log(\varepsilon^{-1}))}$
    be the ReLU $\varepsilon$-approximation of $\psi_\lambda$
    from \Cref{lemma:wavsd}.
    \alert{Let
    $\Phi^\varepsilon_N\in\relu^{d,1}_{CN(1+\log(\varepsilon^{-1}))}$
    be a sum network implemented as in \Cref{thm:calc} (ii) such
    that}
    \begin{align*}
        \alert{
        \rlz(\Phi_N^\varepsilon)=\sum_{\lambda\in\Lambda_N}c_{\lambda,p}(f)
        \rlz(\Phi_\lambda^\varepsilon).
        }
    \end{align*}
    \alert{From \Cref{thm:calc} (ii), the depth of $\Phi^\varepsilon_N$
    is the same as that of the $\Phi_\lambda^\varepsilon$'s.}
    Then,
    \begin{align*}
        \norm{f_N-\rlz(\Phi^{\alert{\varepsilon}}_N)}[\alert{L^p(\Omega)}]\leq
        \varepsilon\sum_{\lambda\in\Lambda_N}|c_{\lambda,p}(f)|.
    \end{align*}
    The sum of the coefficients is bounded by a Besov semi-norm
    of $f$ as we show next.
    
    \alertt{We use the renormalization relationship from
    \eqref{eq:renorm}.}
    \alert{First, we renormalize the coefficients in $L^\tau$, multiply by one and split the
    sum in two
    \begin{align*}
        \sum_{\lambda\in\Lambda_N}|c_{\lambda,p}(f)|
        &=\sum_{\lambda\in\Lambda_N}|c_{\lambda,\tau}(f)|
        2^{\alpha|\lambda|}2^{-|\lambda|(\alpha-d[1/\tau-1/p])}\\
        &=\sum_{\lambda\in\Lambda_N,\;|\lambda|\geq 0}|c_{\lambda,\tau}(f)|
        2^{\alpha|\lambda|}2^{-|\lambda|(\alpha-d[1/\tau-1/p])}\\
        &+
        \sum_{\lambda\in\Lambda_N,\;|\lambda|<0}|c_{\lambda,\tau}(f)|
        2^{\alpha|\lambda|}2^{-|\lambda|(\alpha-d[1/\tau-1/p])}.
    \end{align*}
    
    We abbreviate $\alertt{\alpha^*}:=d[1/\tau-1/p]$, where
    by assumption $0<\alertt{\alpha^*}\leq\alpha$.}    
    Next, if $\tau<1$,
	\alertt{we can estimate the first summand as
	\begin{align*}
		\sum_{\lambda\in\Lambda_N,\;|\lambda|\geq 0}|c_{\lambda,\tau}(f)|
        2^{\alpha|\lambda|}2^{-|\lambda|(\alpha-d[1/\tau-1/p])}
        &=\sum_{\lambda\in\Lambda_N,\;|\lambda|\geq 0}|c_{\lambda,\tau}(f)|
        2^{\alpha^*|\lambda|}\\
        &\leq
        \left(\sum_{\lambda\in\Lambda_N}|c_{\lambda,\tau}(f)|^\tau
        2^{\alpha^*|\lambda|\tau}\right)^{1/\tau},
	\end{align*}
	and the second summand as
	\begin{align*}
		\sum_{\lambda\in\Lambda_N,\;|\lambda|<0}|c_{\lambda,\tau}(f)|
        2^{\alpha|\lambda|}2^{-|\lambda|(\alpha-d[1/\tau-1/p])}
        &\leq
        \sum_{\lambda\in\Lambda_N,\;|\lambda|<0}|c_{\lambda,\tau}(f)|
        2^{-|\lambda|\alpha}\\
        &\leq
        \left(\sum_{\lambda\in\Lambda_N}|c_{\lambda,\tau}(f)|^\tau
        2^{|\lambda|\alpha\tau}\right)^{1/\tau},
	\end{align*}}	    
    \alertt{and thus overall}
    \begin{align*}
        \sum_{\lambda\in\Lambda_N}|c_{\lambda,p}(f)|&\leq
        \left(\sum_{\lambda\in\Lambda_N}|c_{\lambda,\tau}(f)|^\tau
        2^{\alertt{\alpha^*|\lambda|}\tau}\right)^{1/\tau}+
        \left(\sum_{\lambda\in\Lambda_N}|c_{\lambda,\tau}(f)|^\tau
        2^{\alpha\alertt{|\lambda|}\tau}\right)^{1/\tau}\\
        &\leq2\snorm{f}[B_\tau^\alpha(L^\tau(\R^d))],
    \end{align*}
    \alertt{where the last inequality is due to
    the characterization of the Besov semi-norm
    from \Cref{thm:char},
    the fact that
    $\alpha^*\leq\alpha$ and
    the definition of the Besov semi-norm
    (see \Cref{sec:besov}).}

    If $\tau\geq 1$, we apply Hölder with $1\leq\bar\tau\leq\infty$ such that
    $1/\tau+1/\bar\tau=1$,
    \alertt{the definition and characterization
    of the Besov semi-norm
    together with $\alpha^*\leq\alpha$,
    and obtain}
    \begin{align*}
        \sum_{\lambda\in\Lambda_N}|c_{\lambda,p}(f)|\leq
        2N^{1/\bar\tau}\snorm{f}[B_\tau^\alpha(L^\tau(\R^d))].
    \end{align*}
    
    Thus, we set
    \alert{either $\varepsilon:=N^{-\alpha/d}/2$
    for $\tau<1$ or $\varepsilon:=N^{-\alpha/d-1/\bar\tau}/2$
    for $\tau\geq 1$}
    and obtain together with
    \Cref{thm:Nterm}
    \begin{align*}
        \norm{f-\rlz(\Phi^{\alert{\varepsilon}}_N)}[\alert{L^p(\R^d)}]\lesssim
        N^{-\alpha/d}\snorm{f}[B^{\alpha}_{\alert{\tau}}(L^{\tau}(\R^d))].
    \end{align*}
    \alert{Due to the definition of the Besov semi-norm,
    it is \alertt{straightforward} to extend this to Besov semi-norms
    $\snorm{f}[B^{\alpha}_{\alert{q}}(L^{\tau}(\R^d))]$
    for any $0<q\leq\tau$ (see also the discussion in
    \Cref{sec:besov}).}
    The complexity of this network can be bounded by
    \alertt{for the case $\tau\geq 1$ as}
    \begin{align*}
        n:=\alert{CN(1+\log(\varepsilon^{-1}))}
        =\alertt{CN(1+[\alpha/d+\alert{1/\bar\tau}]\log(N))}\lesssim N^{1+\delta},
    \end{align*}
    for any $\delta>0$, or, equivalently,
    \begin{align*}
        \norm{f-\rlz(\Phi_N)}[\alert{L^p(\R^d)}]\lesssim
        n^{-\bar\alpha/d}\snorm{f}[B^{\alpha}_{q}(L^{\tau}(\R^d))],
    \end{align*}
    for any $0<\bar\alpha<\alpha$.
	\alertt{The bound for the case $\tau<1$ is similar, omitting
	$\bar\tau$.}    
    This shows the statement for
    $\Omega=\R^d$.
    
    For $\Omega\subset\R^d$ \alert{an $(\varepsilon,\delta)$-domain
    for $0<\tau\leq 1$ or a locally Lipschitz domain
    for $1<\tau\leq\infty$},
    we use the extension operator from \Cref{thm:extension}
    to obtain for any $f\in B^{\alpha}_{q}(L^{\tau}(\Omega))$
    \begin{align*}
        E(f, \relu_n^{d,1})_{\alert{L^p(\Omega)}}&\leq
        E(\mc Ef, \relu_n^{d,1})_{\alert{L^p(\R^d)}}\lesssim
        n^{-\bar\alpha/d}\snorm{\mc Ef}[B^{\alpha}_{q}(L^{\tau}(\R^d))]
        \\
        &\lesssim n^{-\bar\alpha/d}
        \norm{f}[B^{\alpha}_{q}(L^{\tau}(\Omega))].
    \end{align*}
\end{proof}

Finally, the direct estimates above immediately imply the main result of this work. 

\begin{theorem}[Direct Embeddings]\label{thm:embedd}
    \alert{Let $X=L^p(\Omega)$, $1<p\leq\infty$,
    $\alpha,\tau,q>0$ and
    \begin{align*}
            \alpha/d\geq 1/\tau-1/p,\quad \alert{0<q\leq\tau\leq p},
    \end{align*}
    and assume $\Omega\subset\R^d$ is an $(\varepsilon,\delta)$-domain
    for $\tau\leq 1$ and a strong locally Lipschitz domain for $\tau>1$.
    Then,}
    \begin{enumerate}[label=(\roman*)]
        \item    for $r\geq 2$
        \alertt{and $0<\theta<\alpha$, $0<\bar q\leq\infty$,}
        the following embeddings hold
        \begin{align*}
            B^\alpha_q(L^\tau(\Omega))&\hookrightarrow
            A^{\alpha/d}_\infty(L^p(\Omega),\repu^{r,d,1}),\\
            (L^p(\Omega), B^\alpha_q(L^\tau(\Omega)))_{\theta/\alpha,\bar q}
            &\hookrightarrow
            A^{\theta/d}_{\bar q}(L^p(\Omega),\repu^{r,d,1}).
        \end{align*}
        \item    For $r=1$
        \alertt{and $0<\theta<\bar\alpha$, $0<\bar q\leq\infty$,
        $0<\bar\alpha<\alpha$},
        the following embeddings hold
        \begin{align*}
            B^\alpha_q(L^\tau(\Omega))&\hookrightarrow
            A^{\bar\alpha/d}_\infty(L^p(\Omega),\relu^{d,1}),\\
            (L^p(\Omega), B^\alpha_q(L^\tau(\Omega)))_{\theta/\bar\alpha,\bar q}
            &\hookrightarrow
            A^{\theta/d}_{\bar q}(L^p(\Omega),\relu^{d,1}).
        \end{align*}
    \end{enumerate}
\end{theorem}

\begin{remark}[\alert{Hardy Spaces $H_p(\R^d)$ and $0<p\leq 1$}]\label{rem:hardy}
    \alert{\Cref{thm:char} provides
    best $N$-term approximation rates for Besov spaces in the Hardy $H_p$-norm
    for $0<p\leq 1$.
    ReLU networks can reproduce \alertt{piecewise} linear one-dimensional wavelet systems
    \emph{exactly} and hence all
    the embeddings from \Cref{thm:embedd} hold for
    $r=1$ and
    the homogeneous Besov space $B_q^\alpha(H_\tau(\R))$
    instead of $B_q^\alpha(L^\tau(\Omega))$
    for $0<p\leq 1$ and $1/\tau-1/p\leq \alpha<2$.}
    
    \alert{For more general results
    one would require estimates as in \Cref{lemma:wavsd}
    in the $H_p$-norm.
    However, compactly supported, continuous, \alertt{piecewise} polynomials
    are not necessarily in $H_p$.
    In order to ensure $\rlz(\Phi_\lambda^\varepsilon)$ is in $H_p$,
    we require additionally that the approximands
    $\rlz(\Phi_\lambda^\varepsilon)$ have $d(1/p-1)$
    vanishing moments for any $\varepsilon$ and $\lambda$.}
    
    \alert{Furthermore, to extend results from $\R^d$ to general
    domains $\Omega$,
    one requires constructing extension operators bounded in the
    homogeneous Besov norm.
    A detailed investigation of the case of Hardy spaces
    is thus beyond the scope of this work.}
\end{remark}

\bibliographystyle{siamplain}
\bibliography{literature}

\begin{thebibliography}{10}

\bibitem{RobertAdams2003}
{\sc R.~Adams and J.~J.~F. Fournier}, {\em {Sobolev Spaces}}, Elsevier LTD,
  Oxford, June 2003.

\bibitem{AffineOptimal}
{\sc H.~Bölcskei, P.~Grohs, G.~Kutyniok, and P.~Petersen}, {\em {Optimal
  Approximation with Sparsely Connected Deep Neural Networks}}, SIAM Journal on
  Mathematics of Data Science, 1 (2019), pp.~8--45.

\bibitem{Cohen2003}
{\sc A.~Cohen}, {\em Numerical Analysis of Wavelet Methods}, Elsevier,
  Amsterdam Boston, 2003.

\bibitem{cdf}
{\sc A.~Cohen, I.~Daubechies, and J.-C. Feauveau}, {\em {Biorthogonal Bases of
  Compactly Supported Wavelets}}, Communications on Pure and Applied
  Mathematics, 45 (1992), pp.~485--560.

\bibitem{Daubechies2019arXiv}
{\sc I.~{Daubechies}, R.~{DeVore}, S.~{Foucart}, B.~{Hanin}, and G.~{Petrova}},
  {\em {Nonlinear Approximation and (Deep) ReLU Networks}}, arXiv e-prints,
  (2019), arXiv:1905.02199, p.~arXiv:1905.02199.

\bibitem{devore_1998}
{\sc R.~A. DeVore}, {\em {Nonlinear Approximation}}, Acta Numerica, 7 (1998),
  p.~51–150.

\bibitem{devore1989optimal}
{\sc R.~A. DeVore, R.~Howard, and C.~Micchelli}, {\em {Optimal Nonlinear
  Approximation}}, Manuscripta mathematica, 63 (1989), pp.~469--478.

\bibitem{devorec}
{\sc R.~A. DeVore, P.~Petrushev, and X.~M. Yu}, {\em {Nonlinear Wavelet
  Approximation in the Space $C(R^d)$}}, in Progress in Approximation Theory,
  A.~A. Gonchar and E.~B. Saff, eds., New York, NY, 1992, Springer New York,
  pp.~261--283.

\bibitem{DeVore1984}
{\sc R.~A. DeVore and R.~C. Sharpley}, {\em {Maximal Functions Measuring
  Smoothness}}, Memoirs of the American Mathematical Society, 47 (1984),
  pp.~0--0, \url{https://doi.org/10.1090/memo/0293}.

\bibitem{devore1993besov}
{\sc R.~A. DeVore and R.~C. Sharpley}, {\em {Besov Spaces on Domains in
  ${R}^d$}}, Transactions of the American Mathematical Society, 335 (1993),
  pp.~843--864.

\bibitem{donovan}
{\sc G.~C. Donovan, J.~S. Geronimo, and D.~P. Hardin}, {\em {Orthogonal
  Polynomials and the Construction of Piecewise Polynomial Smooth Wavelets}},
  SIAM Journal on Mathematical Analysis, 30 (1999), pp.~1029--1056.

\bibitem{Dung1996}
{\sc D.~Dung and V.~Q. Thanh}, {\em On nonlinear
  {\textdollar}n{\textdollar}-widths}, Proceedings of the American Mathematical
  Society, 124 (1996), pp.~2757--2765,
  \url{https://doi.org/10.1090/s0002-9939-96-03337-0}.

\bibitem{Fefferman1972}
{\sc C.~Fefferman and E.~M. Stein}, {\em {$H^p$ Spaces of Several Variables}},
  Acta Mathematica, 129 (1972), pp.~137 -- 193,
  \url{https://doi.org/10.1007/BF02392215},
  \url{https://doi.org/10.1007/BF02392215}.

\bibitem{gribonval:hal-02117139}
{\sc R.~{Gribonval}, G.~{Kutyniok}, M.~{Nielsen}, and F.~{Voigtlaender}}, {\em
  {Approximation Spaces of Deep Neural Networks}}, arXiv e-prints,  (2019),
  arXiv:1905.01208, p.~arXiv:1905.01208.

\bibitem{review}
{\sc I.~{G{\"u}hring}, M.~{Raslan}, and G.~{Kutyniok}}, {\em {Expressivity of
  Deep Neural Networks}}, arXiv e-prints,  (2020), arXiv:2007.04759,
  p.~arXiv:2007.04759.

\bibitem{Jones1981}
{\sc P.~W. Jones}, {\em {Quasiconformal Mappings and Extendability of Functions
  in Sobolev Spaces}}, Acta Mathematica, 147 (1981), pp.~71--88,
  \url{https://doi.org/10.1007/bf02392869}.

\bibitem{Kyriazis1996Wavelet}
{\sc G.~C. Kyriazis}, {\em {Wavelet Coefficients Measuring Smoothness in
  $H_p(R^d)$}}, Applied and Computational Harmonic Analysis, 3 (1996),
  pp.~100--119.

\bibitem{SchwabFEM}
{\sc J.~A.~A. Opschoor, P.~C. Petersen, and C.~Schwab}, {\em {Deep ReLU
  Networks and High-order Finite Element Methods}}, Analysis and Applications,
  (2020).

\bibitem{OSZ19_839}
{\sc J.~A.~A. Opschoor, C.~Schwab, and J.~Zech}, {\em {Exponential ReLU DNN
  Expression of Holomorphic Maps in High Dimension}}, Tech. Report 2019-35,
  Seminar for Applied Mathematics, ETH Z{\"u}rich, Switzerland, 2019.

\bibitem{PETERSEN2018296}
{\sc P.~Petersen and F.~Voigtlaender}, {\em {Optimal Approximation of Piecewise
  Smooth Functions Using Deep ReLU Neural Networks}}, Neural Networks, 108
  (2018), pp.~296 -- 330.

\bibitem{Stein70}
{\sc E.~M. Stein}, {\em {Singular Integrals and Differentiability Properties of
  Functions}}, Princeton University Press, 1970,
  \url{http://www.jstor.org/stable/j.ctt1bpmb07}.

\bibitem{suzuki2018adaptivity}
{\sc T.~Suzuki}, {\em {Adaptivity of Deep Re{LU} Network for Learning in Besov
  and Mixed Smooth Besov Spaces: Optimal Rate and Curse of Dimensionality}}, in
  International Conference on Learning Representations, 2019.

\bibitem{YAROTSKY2017103}
{\sc D.~Yarotsky}, {\em {Error Bounds for Approximations with Deep ReLU
  Networks}}, Neural Networks, 94 (2017), pp.~103 -- 114.

\end{thebibliography}

\end{document}